\documentclass[10pt]{amsart} 

\usepackage{a4wide}
\usepackage{geometry}
\usepackage[utf8]{inputenc}
\usepackage{amssymb,amsmath,amscd,amsfonts,amsthm,bbm,mathrsfs,enumerate}
\usepackage{booktabs}
\usepackage[colorlinks=true, linkcolor=blue, citecolor=blue]{hyperref}
\usepackage{mathabx}
\usepackage{graphicx}
\usepackage{color}
\usepackage{accents}
\usepackage{subcaption}
\usepackage{enumitem}
\usepackage{amsaddr}

\DeclareUnicodeCharacter{00A0}{~}

\theoremstyle{definition}
\newtheorem{theorem}{Theorem}
\newtheorem*{theorem*}{Statement}
\newtheorem{lemma}[theorem]{Lemma}

\newtheorem{proposition}[theorem]{Proposition}

\newtheorem{remark}{Remark}

\newtheorem*{condition*}{Condition}

\DeclareMathOperator{\tr}{tr}

\DeclareMathOperator{\RPC}{RPC} 
\DeclareMathOperator{\fop}{fop} 
\DeclareMathOperator*{\diag}{diag}
\DeclareMathOperator*{\lev}{lev}

\newcommand{\normiii}[1]{{\left\vert\kern-0.25ex\left\vert\kern-0.25ex\left\vert #1 
   \right\vert\kern-0.25ex\right\vert\kern-0.25ex\right\vert}}

\newcommand{\1}{\mathbbm 1}
\newcommand{\T}{\top}

\newcommand{\PP}{{{\mathbb P}}} 
\newcommand{\EE}{{{\mathbb E}}} 
\newcommand{\NN}{{{\mathbb N}}} 
\newcommand{\RR}{{{\mathbb R}}} 
 
\newcommand{\SSS}{{{\mathbb S}}}

\newcommand{\mcH}{{\mathscr H}}

\newcommand{\cA}{{\mathcal A}} 
 
\newcommand{\cC}{{\mathcal C}} 
\newcommand{\cD}{{\mathcal D}} 
\newcommand{\cE}{{\mathcal E}}

\newcommand{\cI}{{\mathcal I}} 
 
\newcommand{\cK}{{\mathcal K}}

\newcommand{\cN}{{\mathcal N}} 
 
\newcommand{\cP}{{\mathcal P}} 
\newcommand{\cR}{{\mathcal R}} 
 
\newcommand{\cS}{{\mathcal S}}

\newcommand{\cX}{{\mathcal X}} 
\newcommand{\cU}{{\mathcal U}} 
\newcommand{\cZ}{{\mathcal Z}}

\newcommand{\ps}[2]{\left( {#1}\cdot {#2} \right)}
\newcommand{\EGibbs}[1]{\left\langle #1 \right\rangle}
\newcommand{\indep}{\perp \!\!\! \perp}

\newcommand{\bs}{\boldsymbol}

\newcommand{\bH}{\bs H} 
 
\newcommand{\bP}{\bs P} 
\newcommand{\bQ}{\bs Q}

\newcommand{\bX}{\bs X} 
 
\newcommand{\bZ}{\bs Z} 
 
\newcommand{\bd}{\bs d} 
\newcommand{\bi}{\bs i} 
\newcommand{\bj}{\bs j} 
 
\newcommand{\bm}{\bs m} 
\newcommand{\bq}{\bs q}

\newcommand{\bz}{\bs z} 

\newcommand{\tx}{\tilde x} 

\newcommand{\Hpert}{H^{\text{pert}}} 
\newcommand{\HEF}{H^{\text{EF}}} 
 
\newcommand{\FEF}{F^{\text{EF}}}

\newcommand{\eqdef}{\triangleq}

\newcommand{\tolong}{\xrightarrow[n\to\infty]{}} 
\newcommand{\toasshort}{\stackrel{\text{as}}{\to}}
\newcommand{\toaslong}{\xrightarrow[n\to\infty]{\text{a.s.}}}

\begin{document}

\title[Spin glass analysis of a Lotka-Volterra SDE]
{Spin glass analysis of the invariant distribution of a Lotka-Volterra SDE 
with a large random interaction matrix} 

\author[Gueddari and Hachem]{Mohammed Younes Gueddari and Walid Hachem} 
\date{\today~(revised version)} 
\address{CNRS, Laboratoire d'informatique Gaspard Monge (LIGM / UMR 8049), \\ 
 Université Gustave Eiffel, ESIEE Paris, France} 
\email{gueddari.mohammedyounes@gmail.com, walid.hachem@univ-eiffel.fr} 

\begin{abstract} 
The generalized Lotka-Volterra stochastic differential equation with a
symmetric food interaction matrix is frequently used to model the dynamics of
the abundances of the species living within an ecosystem when these
interactions are mutualistic or competitive. In the relevant cases of interest,
the Markov process described by this equation has an unique invariant
distribution which has a Hamiltonian structure.  Following an important trend
in theoretical ecology, the interaction matrix is considered in this paper as a
large random matrix. In this situation, the (conditional) invariant
distribution takes the form of a random Gibbs measure that can be studied
rigorously with the help of spin glass techniques issued from the field of
physics of disordered systems.  Considering that the interaction matrix is an
additively deformed GOE matrix, which is a well-known model for this matrix in
theoretical ecology, the free energy of the model is derived in the limit of
the large number $n$ of species, making rigorous some recent results from the
literature. The free energy analysis made in this paper could be adapted to 
other situations where the Gibbs measure is non compactly supported. 
\end{abstract} 

\maketitle


\section{Introduction} 

The (generalized) Lotka-Volterra (LV) Stochastic Differential Equation (SDE) is
a standard mathematical model for studying the population dynamics of
biological ecosystems. Letting the integer $n > 0$ be the number of living
species coexisting within an ecosystem, and writing $\RR_+ = [0,\infty)$, the
time evolution of the abundances of these species, \emph{i.e.}, the biomasses
or the numbers of individuals after an adequate normalization, is represented
by the random function $x : \RR_+ \to \RR_+^n$ provided as the solution of the
LV SDE 
\begin{equation}
\label{eds} 
d x_t = x_t \left( 1 + \left( \Sigma - I \right) x_t \right) dt + \phi dt 
   + \sqrt{2 T x_t} dB_t , 
\end{equation} 
with the following notational conventions: $1$ is the $n\times 1$ vector of
ones. Given two $\RR^n$--valued vectors $x = [x_i]$ and $y = [y_i]$ and a
function $f : \RR\to\RR$, we denote as $xy$ the $\RR^n$--valued vector $[x_i
y_i]$ (similarly, $x/y = [x_i/y_i]$ in what follows), and we denote as $f(x)$
the $\RR^n$--valued vector $[f(x_i)]$. When appropriate, a scalar $s$ is
understood as the vector $s 1$.  

In Equation~\eqref{eds}, the $n\times n$ matrix $\Sigma$ is called the food
interaction matrix among the species, the scalar $\phi \geq 0$ is the
immigration rate towards the ecosystem, $B_t \in \RR^n$ is a standard
multi-dimensional Brownian Motion (BM) representing a noise, and $T > 0$ is the
noise temperature. We further assume that $x_0$ is a random variable supported
by $\RR_+^n$ and independent of the Brownian motion $B$.  The model $\sqrt{2T
x_t}$ for the diffusion that we consider here is the so-called demographic
noise model.  Ecological justifications of Equation~\eqref{eds} can be found
in, \emph{e.g.},
\cite{bir-bun-cam-18,roy-etal-19,alt-roy-cam-bir-21,akj-etal-24}.
In all this paper, the interaction matrix $\Sigma$ is assumed symmetric.  This
class of interaction matrices is frequently considered to model the mutualistic
and the competitive interactions \cite{all-tan-12,bir-bun-cam-18,akj-etal-24}.

Recently, the SDE model has raised the interest of the physicists and the
researchers working in the field of random disordered systems, focusing on the
invariant distribution (when it exists) of the SDE~\eqref{eds} seen as a Markov
process.  The framework for this analysis can be described as follows.  When
the dimension of the system, \emph{i.e.}, the number of species, is large, a
large random matrix model is frequently advocated to represent the interaction
matrix $\Sigma$.  This follows a long tradition in theoretical ecology where
the fine structure of the interaction matrix cannot be known, and is replaced
by a random model.  The more or less sophisticated random models used to
represent this matrix aim to better understand the impact of the main
ecological phenomena (competition for the resources, mutualism, predation,
parasitism, ...) that govern the concrete dynamic behavior of these ecosystems
in situations where they contain a large number of species.  When $\Sigma$ is
random, the invariant measure of the Markov process~\eqref{eds} becomes a
conditional (Gibbs) measure which is reminiscent of the Gibbs measure that
appears in the celebrated Sherrington-Kirkpatrick (SK) model for the spin
glasses in the mean field regime. In order to study the asymptotics of this
Gibbs measure as $n\to\infty$, the first step is to compute the asymptotics of
the free energy of the system, as was done by Biroli
\emph{et.al.}~\cite{bir-bun-cam-18}, and Altieri
\emph{et.al.}~\cite{alt-roy-cam-bir-21}.  By using the replica method in the
line of the celebrated papers \cite{par-79,par-80},  they derived the limit
free energy, which, similarly to the well-known SK case, involves a Parisi
probability measure that captures the distribution of the overlaps between the
replicas.  These replica--based computations are used to characterize the
structure of the energy landscape in terms of the temperature $T$ and the
parameters of the random matrix model for the interaction matrix. 

The present paper is a first step towards making rigorous the analyses made in
\cite{bir-bun-cam-18,alt-roy-cam-bir-21}, and the subsequent papers.  Before
delving into the statistical physics, the first part of this paper consists in
a complete analysis of the SDE in terms of existence, uniqueness, and
non-explosion of its solution when $\Sigma$ is deterministic. The existence of
an unique invariant distribution for the Markov process determined
by~\eqref{eds}, and the structure of this distribution are also studied.  For a
symmetric $n\times n$ matrix $A$, define $\lambda_+^{\max}(A)$ as 
\[
\lambda_+^{\max}(A) = \max_{u\in\SSS^{n-1}_+} u^\T A u , 
\]
where $\SSS^{n-1}_+ = \{ u \in \RR_+^n, \ \| u \| = 1 \}$ with $\|\cdot\|$ 
being the Euclidean norm. We show that when the condition 
\begin{equation}
\label{l+}
\lambda_+^{\max}(\Sigma) < 1  
\end{equation} 
is satisfied, then, the LV SDE~\eqref{eds} has an unique strong solution which
is well-defined on $\RR_+$. Furthermore, when $T < \phi$, the Markov process
issued from this SDE has an unique invariant measure, and this invariant
measure is given as $G(x) \sim \exp( \mcH(x) / T )$ where the Hamiltonian
$\mcH$ is given in Equation~\eqref{ham} below.  In the context of our
SDE~\eqref{eds}, much of these results are scattered in the literature,
however, mostly without a rigorous proof up to our knowledge. This will be
the object of Section~\ref{sec-eds}. 

In Section~\ref{sec-l+} and In the remainder of the paper, we assume that
our interaction matrix $\Sigma$ is a symmetric random matrix which is
independent of the BM $B$ and of the initial value $x_0$. Recall that a random
matrix $W_n\in\RR^{n\times n}$ is said to belong to the Gaussian Orthogonal
Ensemble (notation $W_n \sim \text{GOE}_n$) if $W_n$ is equal in law to $(M +
M^\T)/\sqrt{2}$ where $M$ is a real random $n \times n$ matrix with independent
standard Gaussian elements. In Section~\ref{sec-l+} and following, we 
redenote $\Sigma$ as $\Sigma_n$ when useful, and assume that 
\begin{equation}
\label{goedef} 
\Sigma_n = \frac{\kappa}{\sqrt{n}} W_n + \alpha \frac{1 1^\T}{n} , 
\end{equation} 
where $W_n \sim \text{GOE}_n$, and $\kappa > 0$ and $\alpha\in\RR$ are
constants. 
This random matrix model has been frequently considered in the field of
theoretical ecology as an academic model for a mutualistic or weakly 
competitive interaction matrix, where $\alpha$ and $\kappa$ represent 
respectively the normalized mean and standard deviation of the interaction 
among two species \cite{all-tan-12,bun-17,bir-bun-cam-18}.  

For this model, $\lambda_+^{\max}(\Sigma_n)$ becomes of course random. However,
it has been shown by Montanari and Richard \cite{mon-ric-16} in another context
that $\lambda_+^{\max}(\Sigma_n)$ converges almost surely in the large
dimensional regime where $n\to\infty$ to a quantity that can be identified as
the solution of a system of equations in $(\kappa, \alpha)$. In
Section~\ref{sec-l+}, we extend the results of \cite{mon-ric-16} to the
situation where $\alpha$ can be negative.  By doing so, we recover the
realizability bound in the large dimensional regime that is shown in
\cite{bir-bun-cam-18} by building on a former result of Bunin in~\cite{bun-17}. 

When $\Sigma_n$ is random, the invariant measure of the Markov process defined
by~\eqref{eds} becomes a conditional probability measure.  In
Section~\ref{FFR}, we provide an asymptotic analysis of the free energy
associated to this Gibbs measure by using the tools of that are now available
in the mathematical physics literature, leading to Theorem~\ref{F->P} which is
the most important result of this paper.  In most of this literature, the
``spins'' are valued on the set $\{-1, 1\}^n$
\cite{tal-livre11-t1,tal-livre11-t2,pan-livre13}, on the unit-sphere
$\SSS^{n-1}$ \cite{tal-(sph)-06}, or on the rectangle $\cK^n$ where $\cK
\subset\RR$ is a compact set \cite{pan-18}.  One difficulty of our analysis
lies in the fact that our Gibbs measure is supported by the non-compact set
$\RR_+^n$. We hope that our approach can be adapted to other situations where
the support of a Gibbs measure is non-compact, and that our results results
open the door to some further mathematical research on the nature of the Parisi
measure that underlies the limit free energy, leading towards the study of the
Hamiltonian landscape as can be found in the physics literature
\cite{bir-bun-cam-18,alt-roy-cam-bir-21,alt-22}.

In all what follows, $C > 0$ is a generic constant that can change from a line
to another.  This constant can depend on the model dimension $n$ in the next
section but not in the following ones. We denote as $\ps{x}{y}$ the inner
product of the vectors $x,y\in\RR^n$. 

\section{The LV SDE analysis} 
\label{sec-eds} 
In all this section, $n$ is fixed, and $\Sigma$ is a deterministic symmetric
$n\times n$ matrix.  We are concerned here with whether the
SDE~\eqref{eds} is well defined as a SDE on $\RR_+=[0,\infty)$ and by the existence of an
unique invariant measure for the continuous time homogeneous Markov process
$(x_t)_{t\geq 0}$ defined by this SDE. 

Given a symmetric $n\times n$ matrix $A$, similarly to the number 
$\lambda_+^{\max}(A)$ defined above, we define $\lambda_+^{\min}(A)$ as 
\[
\lambda_+^{\min}(A) = \min_{u\in\SSS_+^{n-1}} u^\T A u .
\]
With this definition, Condition \eqref{l+} is equivalent to 
$\lambda_+^{\min}(I - \Sigma) > 0$. In the remainder of the paper, we 
denote as $\lambda_+^{\min} = \lambda_+^{\min}(I-\Sigma)$ this quantity. 

We start with the following proposition, which is proven in 
Appendix~\ref{anx-eds}. 
\begin{proposition}
\label{prop-eds}
Assume that Condition~\eqref{l+} is satisfied. Then, for each initial
probability measure $\mu$ such that $x_0 \sim \mu$, the SDE~\eqref{eds} admits
an unique strong solution on $\RR_+$. 
\end{proposition} 

We shall denote hereinafter as $x^{x_0}_t$ the solution of the SDE~\eqref{eds}
that starts from $x_0$. We have the following proposition: 
\begin{proposition}  
\label{erg} 
Assume that Condition~\eqref{l+} is satisfied. Assume furthermore that 
\begin{equation} 
\label{l-T} 
T < \phi .
\end{equation} 
Then the Markov process $(x_t)$ has an unique invariant distribution
$G(dx) \in \cP(\RR_+^n)$ given as 
\[
G(dx) = \frac{e^{\mcH(x) / T}}{\cZ} dx,  
\]
where $\mcH : (0,\infty)^n \to \RR$ is the Hamiltonian 
\begin{equation}
\label{ham}
\mcH(x) = \frac 12 x^\T \left( \Sigma - I \right) x + (1 \cdot x) 
   + (\phi - T) (1 \cdot \log x)   
\end{equation} 
and $\cZ = \int_{\RR_+^n} \exp(\mcH(x) / T) \, dx < \infty$. Furthermore, for 
each function $\varphi$ integrable with respect to $G(dx)$ and each initial 
value $x_0 \in \RR_+^n$, it holds that 
\[
\frac 1T \int_0^T \varphi(x^{x_0}_t) \ dt 
\xrightarrow[T\to\infty]{\text{a.s.}} \int_{\RR_+^n} \varphi(y) G(dy) 
\] 
\end{proposition} 

To prove this result, we rely on a recurrence result that appears in the
book of Khasminskii \cite{kha-livre12}, see also, \emph{e.g.}, \cite{mao-11}: 
\begin{proposition}
\label{khas} 
[application of Th.~4.1 and 4.2 and Cor.~4.4 of \cite{kha-livre12}]  
The Markov process $(x_t)$ has an unique invariant distribution that we denote
as $G$ if there exists a bounded open domain $\cD \subset (0,\infty)^n$, with 
a regular boundary and a closure $\bar\cD \subset (0,\infty)^n$, that 
satisfies the following property.  For each deterministic 
$x_0 \in \RR_+^n\setminus\cD$, let 
\[
\tau^{x_0}_\cD = \inf \{ t \geq 0 \, : \, x^{x_0}_t \in \cD \} 
\]
be the entry time of $x^{x_0}$ in $\cD$ with $\inf\emptyset=\infty$. For each 
compact set $\cK \subset \RR_+^n$, it holds that 
$\sup_{{x_0}\in \cK} \EE \tau^{x_0}_\cD < \infty$. 

Furthermore, if such a set $\cD$ exists, the convergence
\begin{equation}
\label{ergo} 
\frac 1T \int_0^T \varphi(x^{x_0}_t) \ dt 
\xrightarrow[T\to\infty]{\text{a.s.}} \int_{\RR_+^n} \varphi(y) G(dy) 
\end{equation} 
holds true for each real function $\varphi$ integrable with respect to $G$. 
\end{proposition} 
\begin{proof}[Proof of Proposition~\ref{erg}] 
The infinitesimal generator $\cA$ of the process $(x_t)$ applied to 
a function $\varphi \in \cC^2_{\text{c}}(\RR^n; \RR)$ is 
\[
\cA \varphi(x) = \ps{\nabla \varphi(x)}
   {\left(x(1 + (\Sigma - I)x) + \phi\right)} 
  + T \tr \nabla^2 \varphi(x) \diag(x) .
\]
Let $\varepsilon > 0$ be a small enough number to be fixed later, and put 
$V(x) = \ps{1}{(x - \log (x+\varepsilon))}$ for $x\in\RR_+^n$. A formal 
application of $\cA$ to $V$ shows that 
\begin{align*}
\cA V(x) &= \ps{\left(1 - \frac{1}{x+\varepsilon}\right)}
 {\left( x + x ((\Sigma - I)x) + \phi\right)} 
  + T \ps{1}{\frac{x}{(x+\varepsilon)^2}} \\
 &= \ps{1}{x} - \ps{x}{(I - \Sigma )x} + \phi n 
   - \ps{1}{\frac{x}{x+\varepsilon}} 
 + \ps{x}{(I - \Sigma )\frac{x}{x+\varepsilon}} 
  - \phi \ps{1}{\frac{1}{x+\varepsilon}} 
  + T \ps{1}{\frac{x}{(x+\varepsilon)^2}} .
\end{align*}
Writing 
$ \ps{1}{x} + \ps{x}{(I - \Sigma)\frac{x}{x+\varepsilon}} \leq C\ps{1}{x}$ 
for some $C > 0$, we obtain that 
\begin{align*}
\cA V(x) &\leq - \ps{x}{(I - \Sigma)x} - (\phi - T) 
   \ps{1}{\frac{1}{x+\varepsilon}} + C\ps{1}{x} + \phi n \\
&\leq - \lambda_+^{\min} \| x \|^2  - (\phi - T)
   \ps{1}{\frac{1}{x+\varepsilon}} + C\ps{1}{x} + \phi n \\
&\leq - 0.5 \lambda_+^{\min} \| x \|^2  - (\phi - T)
   \ps{1}{\frac{1}{x+\varepsilon}} + C' 
\end{align*}
for some $C' > 0$ which does not depend on $\varepsilon$. For 
$x = [ x_i ]_{i\in[n]} \in \RR_+^n$, write $x_{\min} = \min_i x_i$. Fix 
$\varepsilon$ as 
\[
\varepsilon = \frac 12 \left( 1 \wedge \frac{\phi - T}{C'+1} \right), 
\]
and define the open domain $\cD \subset \RR_+^n$ as 
\[
\cD = \left\{ x \in \RR_+^n \ : \ 
  \| x \|^2 < \frac{2(1+C')}{\lambda_+^{\min}} 
  \ \text{and} \ x_{\min} > \frac{\phi - T}{C'+1} - \varepsilon \right\} 
\]
(up to enlarging $C'$ if necessary, we can consider that $\cD\neq\emptyset$). 
We observe that $V > 0$ on $\RR_+^n$ and that $\cD$ satisfies the statement of 
Proposition~\ref{khas}.  Furthermore, one can check that for
our choice of $\cD$, it holds that $\cA V(x) \leq -1$ for each 
$x \in \RR_+^n \setminus \cD$. 

Let $x_0$ be a constant vector in $\RR_+^n\setminus \cD$. 
For $a > 0$, define the stopping time $\eta^{x_0}_a = \inf \{ t \geq
0 \ : \ \| x^{x_0}_t \| \geq a \}$.  By Itô's formula, we have by the 
derivation we just made that 
\[
\EE V(x^{x_0}_{t\wedge\tau^{x_0}_\cD\wedge\eta^{x_0}_a}) = 
V({x_0}) 
+ \EE\int_0^{t\wedge\tau^{x_0}_\cD\wedge\eta^{x_0}_a} \cA V(x^{x_0}_u) du 
\leq V({x_0}) - \EE[t\wedge\tau^{x_0}_\cD\wedge\eta^{x_0}_a] . 
\]
Since $V > 0$ on $\RR_+^n$, we obtain that 
$\EE[t\wedge\tau^{x_0}_\cD\wedge\eta^{x_0}_a] \leq V({x_0})$. 
Taking $a\to\infty$ then $t \to\infty$, we get that 
$\EE \tau^{x_0}_\cD < V(x_0)$, which implies that 
$\sup_{{x_0}\in \cK} \EE \tau^{x_0}_\cD < \infty$ for every compact set 
$\cK \subset \RR_+^n\setminus\cD$, and the condition provided on the statement
of Proposition~\ref{khas} is satisfied. Thus, the Markov process defined 
by~\eqref{eds} has an unique invariant distribution $G$, and the 
convergence~\eqref{ergo} holds true. 

It remains to show that the invariant distribution $G$ takes the form provided 
in the statement of Proposition~\ref{erg}. To that end, it is enough to show 
that for $G$ taking this form, it holds that 
\begin{equation}
\label{muinv} 
\forall \varphi \in \cC^2_{\text{c}}(\RR_+^n; \RR), \quad 
\int_{\RR_+^n} \cA \varphi(y) \ G(dy) = 0 . 
\end{equation} 
By an Integration by Parts, we have 
\begin{align*} 
\int_{\RR_+^n} \cA \varphi(y) e^{\mcH(y)/T} \ dy &= 
\int_{\RR_+^n} \left\{ \ps{\nabla\varphi(y)}
 {(y(1 + \Sigma y - y) +\phi)} +
 T \tr \nabla^2\varphi(y) \diag(y) \right\}  e^{\mcH(y)/T} \ dy \\ 
&= \int \varphi(y) \left\{ 
{\nabla}\cdot{\left( (y (-1 - \Sigma y + y) - \phi) 
    e^{\mcH(y)/T}\right)} 
+ T \sum_i \frac{\partial^2}{\partial y_i^2} \left( y_i e^{\mcH(y)/T} \right)  
\right\} \ dy . 
\end{align*} 
We therefore need to show that the term between $\{\cdot\}$ in the integrand 
above is zero for each $y\in(0,\infty)^n$. It is enough to show that 
\[
\forall i \in [n], \quad 
\left( y_i (- 1 - [\Sigma y]_i + y_i) - \phi \right) e^{\mcH(y)/T} 
 = 
 - T \frac{\partial}{\partial y_i} \left( y_i e^{\mcH(y)/T} \right). 
\]
This is directly obtained by developing the right hand side of this expression 
with the help of~\eqref{ham}. 
\end{proof} 

In the remainder of this paper, we shall assume that $T < \phi$. 

\section{Realizability bound for deformed GOE interaction matrix} 
\label{sec-l+} 

From now on, we assume that our interaction matrix $\Sigma_n$ is a large random
$n\times n$ matrix described by Equation~\eqref{goedef}, i.e.
\[
\Sigma_n = \frac{\kappa}{\sqrt{n}} W_n + \alpha \frac{1 1^\T}{n}.
\]
The purpose of this section is to characterize the behavior of
$\lambda_+^{\max}(\Sigma_n)$ as $n\to\infty$ for this model. Recalling that the
$\lambda_+^{\max}(\Sigma_n) = \max_{u\in\SSS^{n-1}_+} u^\T \Sigma_n u$, it is
first interesting to say a few words about the differences between this
quantity and the largest eigenvalue $\lambda^{\max}(\Sigma_n)$ of $\Sigma_n$.
Due to the constraint in the construction of $\lambda_+^{\max}$, it is obvious
that $\lambda_+^{\max}(\Sigma_n) \leq \lambda^{\max}(\Sigma_n)$. As a matter of
fact, these quantities behave quite differently in the large--$n$ regime.  For
instance, when $\alpha = 0$, Lemma~\ref{pt-extrm} below shows that
$\lambda^{\max}_+(\Sigma_n) \toasshort \kappa\sqrt{2}$, while it is well known
from the random matrix theory that $\lambda^{\max}(\Sigma_n) \toasshort
2\kappa$, which is the right edge of a semi-circular law.  The asymptotic
behavior of $\lambda^{\max}(\Sigma_n)$ is obtained with the help of well-known
spectral methods such as the tools based on the resolvent of $\Sigma_n$.  Other
kinds of tools are needed to control the behavior of
$\lambda^{\max}_+(\Sigma_n)$.  Following the approach of Montanari and
Richard~\cite{mon-ric-16}, a lower bound is obtained with the help of an
Approximate Message Passing (AMP) algorithm, while the (matching) upper bound
is obtained with the help of a Sudakov-Fernique inequality coupled with
Gaussian concentration. We shall follow this approach. We must however note
that the authors of~\cite{mon-ric-16} only considered the case $\alpha\geq 0$.
Here, we extend their analysis to the case where $\alpha$ can be negative, and
we correct a minor error in their proof. 

In the remainder of this paper, we shall often drop the index $n$ 
from objects like $\Sigma_n$ or $W_n$ for readability. 

The main result of this section is the following proposition, which should be
compared with~\cite[Theorem 5]{mon-ric-16} establishing the asymptotic behavior
of $\lambda_+^{\max}$ in the case of the so-called symmetric spiked model.
\begin{proposition}
\label{bunin-bd} 
Define the real functions $d, f$, and $g$ on $\RR$ as 
\[
d(x) = \EE\left(Z + x\right)_+^2, \quad 
f(x) = \frac{\EE\left(Z + x\right)_+}{\sqrt{d(x)}},\quad \text{and} \quad 
 g(x) = \frac{\EE Z\left(Z + x\right)_+}{\sqrt{d(x)}},
\]
where $Z\sim \cN(0, 1)$. Then for each $\alpha$, define the function $r_\alpha$ as 
\begin{equation}
    r_\alpha(x) = \alpha f(x)^2 + 2g(x).
    \label{eq:ralpha}
\end{equation}
Then, it holds that 
\[
\lambda_+^{\max}(\Sigma_n) \toaslong \bs\lambda_+(\alpha,\kappa),
\]
where
\[
\bs\lambda_+(\alpha,\kappa) = r_\alpha(c), 
\]
and $c\in\RR$ is the unique solution of the equation 
\[
c = \frac{\alpha}{\kappa} f(c). 
\]
\end{proposition} 
\begin{remark}
We observe that if $\bs\lambda_+(\alpha,\kappa) < 1$, then on the probability
one set where $\lambda_+^{\max}(\Sigma_n) \to \bs\lambda_+(\alpha,\kappa)$, the
function $\exp(\mcH(x) / T)$ is integrable on $\RR_+^n$ for all large $n$, and
the distribution $G$ is well-defined. Alternatively, if
$\bs\lambda_+(\alpha,\kappa) > 1$, then on the probability one set where
$\lambda_+^{\max}(\Sigma_n) \to \bs\lambda_+(\alpha,\kappa)$, for all large
enough $n$, the function $\exp(\mcH(x) / T)$ is not integrable on $\RR_+^n$.
The ``realizability frontier'' $\{ (\alpha,\kappa) :
\bs\lambda_+(\alpha,\kappa) = 1 \}$ which is plotted on Figure~\ref{bun-crv}
coincides with the curve provided in~\cite{bun-17}. 
\end{remark} 

\begin{figure}[h] 
\includegraphics[width=\linewidth]{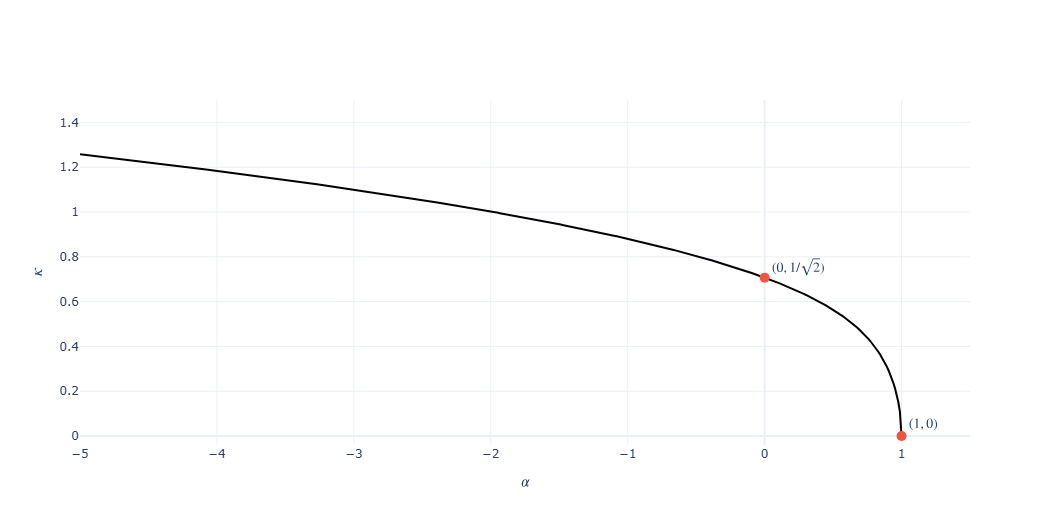}
\caption{The curve $\bs\lambda_+(\alpha,\kappa) = 1$}
\label{bun-crv} 
\end{figure}

Two particular cases deserve some attention:
\begin{lemma}
\label{pt-extrm}
$\bs\lambda_+(0,\kappa) = \kappa\sqrt{2}$ and 
$\lim_{\kappa\to 0} \bs\lambda_+(\alpha,\kappa) = 0 \vee \alpha$. 
\end{lemma} 
This lemma explains the presence of the points $(\alpha,\kappa) =
(0,1/\sqrt{2})$ and $(\alpha,\kappa) = (1,0)$ on the curve of
Figure~\ref{bun-crv}. 

\subsection*{Modifications of the approach of \cite{mon-ric-16} to prove 
Proposition~\ref{bunin-bd}}  

Noticing that 
\[
\lambda_+^{\max}(\Sigma_n) = 
\kappa \lambda_+^{\max}\left(\frac{1}{\sqrt{n}} W_n  
   + \frac \alpha\kappa \frac{1 1^\T}{n} \right)
\quad \text{and} \quad 
\bs\lambda_+(\alpha,\kappa) = \kappa \bs\lambda_+(\alpha/\kappa, 1), 
\]
it is enough to establish Proposition~\ref{bunin-bd} for $\kappa = 1$ and 
$\alpha$ arbitrary. 

We first focus on showing that $\bs\lambda_+(\alpha,1)$ is well-defined as provided
in the statement of the proposition, and provide a maximality property related
with this quantity.  To this end, we adapt the results of \cite[lemmas 18, 20,
22]{mon-ric-16} to the cases where $\alpha$ can be negative. The following
lemma is proven in Appendix~\ref{prf-l-wd}. 
\begin{lemma}
\label{l-wd}
For each $\alpha\in\RR$, the function $r_\alpha : \RR\to\RR$ defined in \eqref{eq:ralpha} as 
\[
r_\alpha(x) = \alpha f(x)^2 + 2 g(x) 
\]
has an unique maximizer $c$, which is defined as the unique solution of the 
fixed point equation 
\[
x = \alpha f(x).
\]  
\end{lemma}

It remains to show that $\lambda_+^{\max}(\Sigma_n) \toasshort
\bs\lambda_+(\alpha, 1) = r_\alpha(c)$ to establish Proposition~\ref{bunin-bd}.

The first step towards showing Proposition~\ref{bunin-bd} 
consists in showing that $\liminf_n \lambda_+^{\max}(\Sigma) \geq 
 \bs\lambda_+(\alpha, 1)$ w.p.1. This is done in \cite{mon-ric-16} 
with the help of an AMP algorithm detailed in their Section IV.A. 
In a word, this iterative algorithm produces a sequence 
$(v^1, v^2, ...)$ of $\RR^n$--valued vectors according to an iteration of the
type
\[
v^{k+1} = \Sigma f(v^k) - \sf b_k f(v^{k-1}) 
\]
where the vector function $f(v)$ and the scalar $\sf b_k$ are provided at 
the beginning of \cite[\S IV.A]{mon-ric-16}. Writing 
$\hat v^k = v^k_+ / \| v^k_+ \|$ where $v^k_+$ is the vector $v^k$ in 
which the negative elements are set to zero, it is shown in 
\cite[Theorem 13]{mon-ric-16} (with our notations) that 
\[
\lim_{k\to\infty} \lim_{n\to\infty} 
 \ps{\hat v^k}{\Sigma \hat v^k} = r_\alpha(c) \quad 
  \text{w.p.} 1, 
\]
which shows that $\liminf_n \lambda_+^{\max}(\Sigma) \geq r_\alpha(c)$ w.p.1. 
This argument can be extended to the case $\alpha\leq 0$ practically 
without modification. 

The second step consists in showing that 
$\limsup_n \lambda_+^{\max}(\Sigma_n) \leq r_\alpha(c)$ w.p.1. 
Introducing the set $\SSS_+^{n-1}(s)$ defined for $s\in [0,1]$ as 
\[
\SSS_+^{n-1}(s) = \left\{ u \in \SSS_+^{n-1}, \ps{u}{1/\sqrt{n}} = s \right\}, 
\]
we can characterize $\lambda_+^{\max}(\Sigma)$ by the identity 
\begin{align*}
\lambda_+^{\max}(\Sigma) &= 
  \max_{s\in[0,1]} \left( \alpha s^2 + \max_{u\in\SSS_+^{n-1}(s)} 
  \ps{u}{W u} \right) \\
 &\eqdef \max_{s\in[0,1]} M(s). 
\end{align*} 
By a Gaussian concentration argument \cite[Appendix B]{mon-ric-16}, it is 
enough to show that 
\begin{equation}
\label{max-M}
\max_{s\in[0,1]} \limsup_n \EE M(s) \leq r_\alpha(c) 
\end{equation} 
to obtain that $\limsup_n \lambda_+^{\max}(\Sigma) \leq r_\alpha(c)$ w.p.1. 
This bound can be obtained by a Sudakov-Fernique argument for Gaussian
processes \cite{ver-livre18}. Consider the Gaussian process 
$u \mapsto \xi_u = 2 \ps{u}{\bZ}$ on $\SSS_+^{n-1}(s)$ with 
$\bZ \sim \cN(0, n^{-1} I_n)$. We can easily show that 
$n^{-1} \EE \left( \ps{u}{Wu} - \ps{v}{W v} \right)^2 \leq 
\EE (\xi_u - \xi_v)^2$.  By the Sudakov-Fernique inequality, we then have 
\begin{align*}  
\EE M(s) &\leq \alpha s^2 + 2\EE \max_{u\in\SSS_+^{n-1}(s)} \ps{u}{\bZ} \\ 
        &\leq \alpha s^2 -2\tilde c s 
 + 2\EE \max_{u\in\SSS_+^{n-1}} \ps{u}{(\bZ + \tilde c 1/\sqrt{n})} \\
        &= \alpha s^2 -2\tilde c s 
              + 2\EE \| (\bZ + \tilde c 1/\sqrt{n})_+ \| \\
&\leq \alpha s^2 -2\tilde c s 
     + 2\left(\EE \| (\bZ + \tilde c 1/\sqrt{n})_+ \|^2\right)^{1/2}  , 
\end{align*} 
where $\tilde c \in \RR$ is arbitrary. Taking $n$ to infinity, we get that
\[
\limsup_n \EE M(s) \leq \alpha s^2 -2\tilde c s + 2\sqrt{d(\tilde c)} .
\]
Let us now take $\tilde c$ as the unique solution of the equation 
$f(x) = s$. Using the identity $\sqrt{d(x)} = x f(x) + g(x)$, we then obtain 
\[
\limsup_n \EE M(s) \leq \alpha f(\tilde c)^2 -2\tilde c f(\tilde c) 
    + 2\sqrt{d(\tilde c)} = r_\alpha(\tilde c) \leq r_\alpha(c) 
\]
by Lemma~\ref{l-wd}, and the bound~\eqref{max-M} is established. 

We note here that an analogous argument is used in \cite{mon-ric-16} to
prove~\cite[Lemma 10]{mon-ric-16}. However, in the context of
Proposition~\ref{bunin-bd} above, this argument requires the function $s\mapsto
\alpha s^2 - 2\tilde c s$ to be concave, which is incorrect when $\alpha > 0$.

The proof of Proposition~\ref{bunin-bd} is complete. 

\subsection*{Proof of Lemma~\ref{pt-extrm}}

We have $\bs\lambda_+(0,\kappa) = 2\kappa g(0) = \kappa\sqrt{2}$ 
by a small calculation. 

To establish the second result for $\alpha\neq 0$, let us focus on the equation
$c = (\alpha/\kappa) f(c)$. By the properties of the function $q$ provided
above, we see that $c  = c(\kappa)$ converges to $\infty$ as $\kappa\to 0$ 
when $\alpha > 0$, and to $- \infty$ when $\alpha < 0$. 
We also know that $g$ is bounded. Therefore, 
$\lim_{\kappa\to 0} \bs\lambda_+(\alpha,\kappa) = \lim_{\kappa\to\infty} 
\alpha f(c(\kappa)) = 0 \vee \alpha$ by recalling the properties of $f$.   

\section{Asymptotics of the free energy} 
\label{FFR}
From now on, we assume that $\kappa$ and $\alpha$ in our deformed GOE model are
chosen in such a way that $\bs\lambda_+(\kappa,\alpha) < 1$, and we set
$\varepsilon_\Sigma = (1 - \bs\lambda_+(\kappa,\alpha)) / 2$.  Since $\Sigma_n$
is now a large random matrix for which $\PP\left[\lambda_+^{\max}(\Sigma_n) >
1\right] > 0$ for any fixed $n$, we need to replace our SDE~\eqref{eds} with
the SDE  
\[
d x_t = x_t \left( 1 + \left( \widetilde\Sigma - I \right) x_t \right) dt 
+ \phi dt + \sqrt{2 T x_t} dB_t 
\]
with 
\[
\widetilde\Sigma = \Sigma 
   \1_{\lambda_+^{\max}(\Sigma) < 1 - \varepsilon_\Sigma} 
\]
to guarantee the solution is well-defined. 
By a slight adaptation of Proposition~\ref{prop-eds} to the case where the
interaction matrix is random, this SDE admits an unique strong solution.
Furthermore, since $1 - \varepsilon_\Sigma = \bs\lambda_+(\kappa,\alpha) +
\varepsilon_\Sigma$, it holds by  Proposition~\ref{bunin-bd} that with
probability one, $\Sigma = \widetilde\Sigma$ for all $n$ large enough. 
We use hereinafter the standard notation in spin glass theory $\beta = 1/T$, 
and recall our  standing assumption $\phi\beta > 1$. 

The invariant distribution $G$ of the EDS~\eqref{eds}, which existence is
ensured by Proposition~\ref{erg}, becomes in our new setting the conditional
distibution 
$\widetilde G(dx) = \widetilde \cZ^{-1} \exp(\beta \widetilde \mcH(x)) dx$ 
given $\widetilde\Sigma$, with 
\[
\widetilde \mcH(x) = 
 \frac 12 x^\T \left( \widetilde\Sigma - I \right) x + (1 \cdot x) 
   + (\phi - 1 / \beta) (1 \cdot \log x),  
\] 
and 
\[
\widetilde \cZ = \int_{\RR_+^n} \, \exp(\beta \widetilde \mcH(x)) \ dx 
             < \infty .
\]
In what follows, we shall use a more convenient expression for the measure
$\exp(\beta\mcH(x)) \, dx$ on $\RR_+^n$ by writing 
\[
\exp(\beta\widetilde\mcH(x)) \, dx = 
 \exp\left( H_n(x) 
  \1_{\lambda_+^{\max}(\Sigma_n) < 1-\varepsilon_\Sigma}\right) 
  \, \mu_\beta^{\otimes n}(dx) 
\]
where $H_n : \RR_+^n \to \RR$ is the new Hamiltonian 
\[
H_n(x) = \frac{\beta}{2} x^\T \Sigma_n x = 
\frac{\beta\kappa}{2\sqrt{n}} x^\T W_n x + 
         \frac{\beta\alpha}{2} \frac{\ps{x}{1}^2}{n}, 
\]
and $\mu_\beta^{\otimes n}(dx)$ is the $n$--fold product measure of the 
positive measure $\mu_\beta$ defined on $\RR_+$ as 
\[
\mu_\beta(dx_1) = x_1^{\phi\beta - 1} 
   \exp\left( -\beta x_1^2 / 2 + \beta x_1 \right) \, d x_1, \quad x_1 \geq 0. 
\]
Our purpose is to identify the asymptotic behavior of the
free energy 
\[
\widetilde F_n = \frac 1n \EE \log \widetilde \cZ 
 = \frac 1n \EE \log \int_{\RR_+^n} \, 
 \exp\left(H_n(x)   \1_{\lambda_+^{\max}(\Sigma_n) < 1-\varepsilon_\Sigma}
  \right) 
  \, \mu_\beta^{\otimes n}(dx) . 
\]
We now construct the Parisi functional that will be used in our main theorem. 
To begin with, we set the notations for defining a the functional order parameter
(f.o.p.). Let $D > 0$, let $K > 0$ be an integer, and let 
\begin{equation}
\label{lq} 
\begin{split} 
& 0 < \lambda_0 < \cdots < \lambda_{K-1} < 1, \quad \text{and}  \\
& 0 = b_0 < \cdots < b_{K-1} < b_K = D.  
\end{split} 
\end{equation} 
We define the associated functional order parameter associated with these sequences by $\zeta$ defined as 
\begin{equation}
\label{fop} 
\zeta(\{b_k\}) = \lambda_k - \lambda_{k-1} \quad \text{with } \quad
 \lambda_{-1} = 0, \ \lambda_K = 1. 
\end{equation} 
We shall denote the set of these finitely supported measures such that the 
maximum of the support is equal to $D$ as $\fop_D \subset \cP([0,D])$. 

Let $a > 0$, $h\geq 0$ and $\gamma\in\RR$, and let $(z_k)_{k\in[K]}$ be $K$ 
independent standard Gaussians. Define the random variable 
\[
X_{K,a} = 
\log\int_0^a 
  \exp\left( x \beta\kappa \sum_{k=1}^K  z_k \sqrt{b_k - b_{k-1}} 
  + \beta\alpha h x + \gamma x^2 \right) \mu_\beta(dx) . 
\] 
For $k = K-1, \ldots, 0$, set 
\[
X_{k,a} = \frac{1}{\lambda_k} 
  \log \EE_{z_{k+1}} \exp\left( \lambda_k X_{k+1} \right) , 
\]
where $\EE_{z_{k+1}}$ is the expectation with respect to the distribution of
$z_{k+1}$. With this, construction, $X_{0,a}$ is deterministic and depends on 
$\zeta$, $h$, and $\gamma$. Denote this real number as 
$X_{0,a}(\zeta,h,\gamma)$. Our Parisi functional will be 
\[
P_a(\zeta,h,\gamma) = 
X_{0,a}(\zeta,h,\gamma) 
- \frac{\beta^2\kappa^2}{4} \sum_{k=0}^{K-1} \lambda_k (b_{k+1}^2- b_{k}^2) . 
\]
We can now state our main result. 
\begin{theorem}
\label{F->P} 
Assume that $\bs\lambda_+(\kappa,\alpha) < 1$. Then, if $\alpha \leq 0$, it
holds that 
\[
\widetilde F_n \tolong 
 \sup_{a \geq 0,D \geq 0} \ \inf_{\zeta\in \fop_D,h \geq 0, \gamma\in\RR} 
  \left( P_a(\zeta,h,\gamma) - \gamma D - \frac{\beta\alpha}{2} h^2 \right).
\] 
If $\alpha > 0$, then 
\[
\widetilde F_n \tolong 
 \sup_{a \geq 0, D \geq 0, h \geq 0} \ \inf_{\zeta\in \fop_D,\gamma\in\RR} 
  \left( P_a(\zeta,h,\gamma) - \gamma D - \frac{\beta\alpha}{2} h^2 \right).
\] 
\end{theorem} 
An equivalent representation of the limit of the free energy is provided by the
following proposition.
\begin{proposition}
\label{P_infty}
For each $a\in(0,\infty)$, the function $P_a(\zeta,h,\gamma)$ can be
continuously extended from $\fop_D \times \RR_+\times\RR$ to
$\cP([0,D])\times\RR_+\times\RR$ under the product of the narrow topology and
the Borel topology on $\RR_+\times\RR$.  Still denote as $P_a(\zeta,h,\gamma)$
the function obtained by this continuous extension, which is increasing in $a$,
and let $P_\infty(\zeta,h,\gamma) = \lim_{a\to\infty} P_a(\zeta,h,\gamma)$ for
$\zeta\in\cP([0,D])$. 

Assume that $\bs\lambda_+(\kappa,\alpha) < 1$. Then, if $\alpha \leq 0$, it
holds that 
\[
\widetilde F_n \tolong 
 \sup_{D \geq 0} \ \inf_{\zeta\in \cP([0,D]),h \geq 0, \gamma\in\RR} 
  \left( P_\infty(\zeta,h,\gamma) - \gamma D - \frac{\beta\alpha}{2} h^2 \right).
\] 
If $\alpha > 0$, then 
\[
\widetilde F_n \tolong 
 \sup_{D \geq 0, h \geq 0} \ \inf_{\zeta\in \cP([0,D]),\gamma\in\RR} 
  \left( P_\infty(\zeta,h,\gamma) - \gamma D - \frac{\beta\alpha}{2} h^2 \right).
\] 
\end{proposition} 

To establish Theorem~\ref{F->P}, we shall follow Panchenko's approach
in~\cite{pan-18}.  The main particularities of our proof are related with 
the non-compactness of the support of $\mu_\beta$ and the presence of the
factor $\alpha$ in the expression of the Hamiltonian $H$. 

A comprehensive analysis of the limits stated by the previous theorem remains
to be done.  We conjecture that these limits are infinite when
$\bs\lambda_+(\kappa,\alpha) > 1$. In the case where
$\bs\lambda_+(\kappa,\alpha) < 1$, a challenge is to prove that each of the
saddles that appear in the expressions of these limits is attained by an unique
probability measure $\zeta_\star$ which is compactly supported. In the physics
literature, the study of the structure of $\zeta_\star$ has attracted a great
deal of interest.  For instance, it is asserted that in the low temperature and
low immigration rate regime $\beta\to\infty$ and $\phi\to 0$, the measure
$\zeta_\star$ is reduced to a Dirac delta when $\kappa < 1/\sqrt{2}$ (see
Figure~\ref{bun-crv}), which corresponds to the so-called Replica Symmetry
regime. For $\kappa > 1/\sqrt{2}$, the measure $\zeta_\star$ has an 
infinite support, corresponding to the so-called Full Replica Symmetry 
Breaking regime \cite{alt-roy-cam-bir-21}. 

We now prove Theorem~\ref{F->P}. 

\section{Proof of Theorem~\ref{F->P} and Proposition~\ref{P_infty}: some 
  preparation} 
\label{prep} 

\subsection{Ruelle probability cascades} 
It is well-known that a Parisi functional is intimately connected with the
so-called Ruelle Probability Cascades (RPC) which definition and properties are
discussed at length in \cite[Chapter~2]{pan-livre13}. We recall the main
properties of these objects for the reader's convenience, and we provide an
equivalent expression of our Parisi functional $P_a(\zeta,h,\gamma)$ involving
RPC's.  Arranging the $\lambda_k$'s in~\eqref{lq} in the vector $\bs\lambda =
(\lambda_0,\ldots,\lambda_{K-1})$, the notation $\RPC_{\bs\lambda}
\in\cP(\cP(\NN^K))$ will denote the distribution of a RPC
$(v_{\bi})_{\bi\in\NN^K}$ with parameters the elements of the vector $\bs
\lambda$. We shall repeatedly use the following result, obtained by a direct
adaptation of \cite[Theorem 2.9]{pan-livre13}.

\begin{proposition}
\label{P-RPC} 
Let $d > 0$ be an integer.  Consider the vector $\bs\lambda$ above, $K$ random
vectors $\bz_1,\ldots,\bz_K \stackrel{\text{iid}}{\sim} \cN(0,I_d)$, and a 
real function $\cX_K(\bz_1,\ldots,\bz_K)$ satisfying $\EE
\exp\lambda_{K-1}\cX_K < \infty$. For $k = K-1$ to $0$, define recursively 
\[
\cX_k = \frac{1}{\lambda_k} \log \EE_{\bz_{k+1}} e^{\lambda_k \cX_{k+1}} , 
\]
where $\EE_{\bz_{k+1}}$ is the expectation with respect to the distribution of
$\bz_{k+1}$.  Consider a family of i.i.d.~random vectors $(\bz_{\bj})_{\bj \in
\NN \cup \NN^2 \cup \cdots \cup \NN^K}$ with distribution $\cN(0, I_d)$. For
$\bi \in \NN^k$, write $\bi = (i_1,i_2,\ldots, i_k)$. Then, it holds that 
\[
\cX_0 = \EE\log \sum_{\bi\in\NN^K} v_{\bi} 
 \exp \cX_K(\bz_{i_1}, \bz_{i_1 i_2}, \ldots, \bz_{i_1,i_2,\ldots, i_K}) ,
\]
where $(v_{\bi})_{\bi\in\NN^K} \sim \RPC_{\bs\lambda}$ is independent of 
the family $(\bz_{\bj})$. 
\end{proposition}

We will also need the following application of this result.  Write $\bz_k =
[z_{k,l}]_{l=1}^d$, and assume that  $\cX_K$ is of the form
$\cX_K(\bz_1,\ldots,\bz_K) = \sum_{l=1}^d \cU_{K}(z_{1,l},\ldots,z_{K,l})$ 
for some real function $\cU_{K}$. Performing the induction 
$\cX_K\to \cX_{K-1}\to \cdots\to \cX_0$ as in the statement of the previous 
proposition, we can check that $\cX_0 = d \cU_0$, where $\cU_0$ is
also obtained from $\cU_K$ by the same kind of induction.

We now provide an expression of our Parisi functional $P_a$ in terms of a 
RPC with the help of the former proposition. As is usual in the SK literature,
we introduce a function $\xi$ capturing the covariance function of the Gaussian
process $x \mapsto \beta\kappa x^\T W x / (2\sqrt{n})$. Namely, we write 
\[
\frac{\beta^2\kappa^2}{4n} \EE (x^1)^\T W x^1 (x^2)^\T W x^2 = 
 n \xi(R_{12}) \quad \text{with} \quad 
 R_{ij} = \frac{\ps{x^i}{x^j}}{n} \quad \text{and} \quad 
\xi(x) = \frac{\beta^2\kappa^2}{2} x^2. 
\]
Recall that the variances on the diagonal of the GOE matrix $W$ are equal to $2$, hence the factor $1/2$ instead of $1/4$ in the definition of $\xi$.

We also find it more readable and more compliant with the SK literature to 
set $\theta(x) = x \xi'(x) - \xi(x)$, even though in our case, this 
function trivially reads $\theta(x) = \xi(x)$. 

With the help of Proposition~\ref{P-RPC}, we can provide an expression of 
the term $X_{0,a}$ in the expression of $P_a$ in terms of a RPC. 
With our new notations, the function $X_{K,a}$ is re-written 
\[
X_{K,a} = \log\int_0^a 
  \exp\left( x \sum_{k=1}^K  z_k \sqrt{\xi'(b_k) - \xi'(b_{k-1})} 
  + \beta\alpha h x + \gamma x^2 \right) \mu_\beta(dx) . 
\] 
Set $d=1$, and let $\cX_K = X_{K,a}$ in Proposition~\ref{P-RPC} with 
$\bz_k = z_k$. Consider a family of independent standard Gaussians 
$(z_{\bj})_{\bj \in \NN \cup \NN^2 \cup \cdots \cup \NN^K}$, and define the 
family of random variables $( q_{\bi})_{\bi\in\NN^K}$ as follows. For 
$\bi = (i_1,i_2,\ldots, i_K) \in \NN^K$, set 
\begin{equation}
\label{qi} 
q_{\bi} = \beta\kappa \sum_{k=1}^K z_{i_1,i_2,\ldots,i_k} 
  \sqrt{b_k - b_{k-1}} 
 = \sum_{k=1}^K z_{i_1,i_2,\ldots,i_k} \sqrt{\xi'(b_k) - \xi'(b_{k-1})}. 
\end{equation} 
Then, by Proposition~\ref{P-RPC} above, we have 
\[
X_{0,a}(\zeta,h,\gamma) = \EE \log \sum_{\bi \in \NN^K} 
 v_{\bi} \int_0^a e^{x q_{\bi} + \beta\alpha h x + \gamma x^2} \mu_\beta(dx) . 
\] 
We now consider the term 
$- (\beta^2\kappa^2 / 4) \sum_{k=0}^{K-1} \lambda_k (b_{k+1}^2- b_{k}^2)$ 
in the expression of $P_a$. First, it will be convenient to write 
\begin{align*} 
 \frac{\beta^2\kappa^2}{4} 
   \sum_{k=0}^{K-1} \lambda_k (b_{k+1}^2- b_{k}^2)
 &=  \frac 12 \sum_{k=0}^{K-1} \lambda_k (\theta(b_{k+1})- \theta(b_{k})) \\
 &= - \frac 12 \sum_{k=0}^K (\lambda_k - \lambda_{k-1}) \theta(b_k) 
 + \frac 12 \theta(b_K)  
 = - \frac 12 \int \theta(b) \zeta(db) + \frac 12 \theta(D). 
\end{align*} 
Second, if we set $d = 1$ and 
\[
\cX_K(\bz_1,\ldots,\bz_K) = 
 \frac{\beta\kappa}{\sqrt{2}} \left( \bz_1 \sqrt{b_1^2 - b_{0}^2} + \cdots 
  + \bz_K \sqrt{b_K^2 - b_{K-1}^2} \right) 
 = \sum_{k=1}^K \bz_k \sqrt{\theta(b_k) - \theta(b_{k-1})}  
\]
in Proposition~\ref{P-RPC}, then, using the expression of the moment 
generating function of a Gaussian, we easily obtain that 
\[
\cX_0 = \frac 12 \sum_{k=0}^{K-1} \lambda_k (\theta(b_{k+1})- \theta(b_{k}))
  = - \frac 12 \int \theta(b) \zeta(db) + \frac 12 \theta(D). 
\]
similarly to the family $( q_{\bi})_{\bi\in\NN^K}$, define the family of 
random variables $( y_{\bi} )_{\bi\in\NN^K}$ as 
\begin{equation}
\label{yi} 
y_{\bi} = \frac{\beta\kappa}{\sqrt{2}} \sum_{k=1}^K z_{i_1,i_2,\ldots,i_k} 
  \sqrt{b_k^2 - b_{k-1}^2} 
 = \sum_{k=1}^K z_{i_1,i_2,\ldots,i_k} \sqrt{\theta(b_k) - \theta(b_{k-1})} .
\end{equation} 
Then, by Proposition~\ref{P-RPC}, we also have 
\[
 \cX_0 = \EE \log \sum_{\bi \in \NN^K} v_{\bi} e^{y_{\bi}} .
\]
In short, we can write 
\begin{align*} 
P_a(\zeta,h,\gamma) &= 
X_{0,a}(\zeta,h,\gamma) 
   + \frac 12 \int \theta(b) \zeta(db) - \frac 12 \theta(D)  \\ 
&= \EE \log \sum_{\bi \in \NN^K} 
 v_{\bi} \int_0^a e^{x q_{\bi} + \beta\alpha h x + \gamma x^2} 
   \mu_\beta(dx) 
 - \EE \log \sum_{\bi \in \NN^K} v_{\bi} e^{y_{\bi}} 
\end{align*}
for $\zeta\in\fop_D$. 

\subsection{Useful bounds} 
We shall prove Theorem~\ref{F->P} by computing in turn an upper bound on 
$\limsup \widetilde F_n$ and a lower bound on $\liminf\widetilde F_n$. 
To this end, we shall rely on the two following lemmas thanks to which 
$\widetilde F_n$ will be replaced with more easily manageable free 
energies. 

Given $A > 0$, let $B_+^n(\sqrt{n} A) = \{ x \in \RR_+^n \ : \ \| x \| \leq
\sqrt{n} A \}$ be the closed Euclidean $\sqrt{n} A$--ball of $\RR_+^n$. 
Let $F^A_n$ be the free energy defined as 
\[
F^A_{n} = \frac 1n \EE \log\int_{B_+^n(\sqrt{n} A)} e^{\beta\mcH(x))} dx 
 = \frac 1n \EE\log \int_{B_+^n(\sqrt{n} A)} 
 e^{H_n(x)} \mu_\beta^{\otimes n}(dx) ,  
\]
where the indicator $\1_{\lambda_+^{\max}(\Sigma) < 1 - \varepsilon_\Sigma}$ 
is absent, but where the integration is performed on $B_+^n(\sqrt{n} A)$. 
The free energy $F^A_n$ is much more easily manageable that $\widetilde F_n$
because $H_n(x)$ is a Gaussian process, contrary to 
$H_n(x) \1_{\lambda_+^{\max}(\Sigma) < 1 - \varepsilon_\Sigma}$. 
The following lemma, proven in Appendix~\ref{prf-F<F^A}, will let us focus  
our upper bound analysis on $F^A_n$:  
\begin{lemma} 
\label{F<F^A} 
There exists $A > 0$ large enough such that 
$\displaystyle{\limsup_n \widetilde F_n \leq \limsup_n F_n^A}$.
\end{lemma} 

For $a > 0$, define now the free energy 
\[
F_{a,n} = \frac 1n \EE \log\int_{[0,a]^n} e^{\beta\mcH(x)} dx 
 = \frac 1n \EE\log \int_{[0,a]^n} e^{H_n(x)} \mu_\beta^{\otimes n}(dx) . 
\]
The following lemma is proven is Appendix~\ref{prf-F>Fa}.  
\begin{lemma}
\label{F>Fa} 
$\liminf_n \widetilde F_n \geq \sup_{a > 0} \liminf_n F_{a,n}$. 
\end{lemma}

\section{Proof of Theorem~\ref{F->P}: upper bound for $\alpha\leq 0$} 
\label{sec:ubnd} 

Our purpose here is to establish the following proposition. 
\begin{proposition}
\label{up-bnd}
Assume that $\alpha \leq 0$. Then, 
\[
\limsup_n \widetilde F_n \leq 
 \sup_{a \geq 0, D\geq 0} \ \inf_{\zeta\in \fop_D, h \geq 0, \gamma\in\RR} 
  \left( P_a(\zeta,h,\gamma) - \gamma D - \frac{\beta\alpha}{2} h^2 \right)  . 
\]
\end{proposition} 

Thanks to Lemma~\ref{F<F^A}, we focus our analysis on $F^A_n$. Our technique
will be based on the so-called Guerra's interpolation. Let $A \in (0,\infty]$
and $a\in (0,\infty]$. Given $D > 0$ and an integer $K > 0$, let
$\zeta\in\fop_D$ be defined as in~\eqref{fop} with the parameters
in~\eqref{lq}. Let $h\geq 0$ and $\gamma\in\RR$, and let $(\bz_k)_{k\in[K]}$ be
$K$ independent $\cN(0,I_n)$ random vectors. Define the random variable 
\[
\bX_{K,a}^A(\bz_1,\ldots, \bz_K) = \log 
 \int_{B_+(A\sqrt{n})\cap[0,a]^n}  
  \exp\left( \ps{x}{\sum_{k=1}^K  \sqrt{\xi'(b_k) - \xi'(b_{k-1})} \bz_k} 
  + \beta\alpha h \ps{1}{x} + \gamma \|x\| ^2 \right) 
   \mu_\beta^{\otimes n} (dx) . 
\] 
Applying Proposition~\ref{P-RPC} with $d = n$ and 
$\cX_K(\bz_1,\ldots,\bz_K) = \bX_{K,a}^A(\bz_1,\ldots, \bz_K)$, the result 
$\cX_0$ of the recursion that we denote as 
$\cX_0 = \bX_{0,a}^A(\zeta,h,\gamma)$ can be written as 
\begin{equation}
\label{X0A-rpc} 
\bX_{0,a}^A(\zeta,h,\gamma) =  \EE \log \sum_{\bi \in \NN^K} 
 v_{\bi} \int_{B_+(A\sqrt{n})\cap[0,a]^n}  
  \exp \left( \ps{x}{\bs q_{\bi}} + \beta\alpha h \ps{1}{x} 
  + \gamma \|x \|^2 \right) 
   \mu_\beta^{\otimes n}(dx) 
\end{equation} 
where the process $(\bq_{\bi})_{\bi\in\NN^K}$ with $\bq_{\bi} = 
 [ \bq_{\bi,l} ]_{\l=1}^n$ is such that the $n$ processes 
$(\bq_{\bi,l})_{\bi\in\NN^K}$ for $l\in[n]$ are $n$ independent copies of 
the process $(q_{\bi})$ in~\eqref{qi}, and the processes $(v_{\bi})$ and
$(\bq_{\bi})$ are independent. 

With this, define 
\[
\bP_a^A(\zeta,h,\gamma) = \frac 1n \bX_{0,a}^A(\zeta,h,\gamma) 
 + \frac 12 \int \theta d\zeta - \frac{\theta(D)}{2} . 
\]
We shall establish the following lemma. 
\begin{lemma}
\label{guerra} For $\alpha\leq 0$, it holds that 
$\displaystyle{\limsup_n \left\{F^A_n - 
  \sup_{D\in[0,A^2]} \inf_{\zeta\in\fop_D, h\in\RR_+,\gamma\in\RR}  
 \left( \bP_{\infty}^A(\zeta, h,\gamma) - \gamma D 
  - \frac{\beta\alpha}{2} h^2 \right)\right\}\leq 0}$. 
\end{lemma} 
Noticing that $\bP_{\infty}^A = \bP_{a}^A$ for $a \geq A \sqrt{n}$ and 
observing that $\bP_{a}^\infty(\zeta, h,\gamma) = P_{a}(\zeta, h,\gamma)$ by
the remark that follows Proposition~\ref{P-RPC}, we obtain
\begin{align*} 
& \sup_{D\in[0,A^2]} \inf_{\zeta\in\fop_D, h\in\RR_+,\gamma\in\RR}  
 \left( \bP_{\infty}^A(\zeta, h,\gamma) - \gamma D 
  - \frac{\beta\alpha}{2} h^2 \right)  \\ 
&= \sup_{D\in[0,A^2]} 
 \inf_{\zeta\in\fop_D, h\in\RR_+,\gamma\in\RR}  
 \left( \bP_{A\sqrt{n}}^A(\zeta, h,\gamma) - \gamma D 
  - \frac{\beta\alpha}{2} h^2 \right)  \\ 
&\leq \sup_{a \geq 0, D\geq 0} 
 \inf_{\zeta\in\fop_D, h\in\RR_+,\gamma\in\RR}  
 \left( \bP_{a}^A(\zeta, h,\gamma) - \gamma D 
  - \frac{\beta\alpha}{2} h^2 \right) \\
&\leq \sup_{a \geq 0, D\geq 0} 
 \inf_{\zeta\in\fop_D, h\in\RR_+,\gamma\in\RR}  
 \left( \bP_{a}^\infty(\zeta, h,\gamma) - \gamma D 
  - \frac{\beta\alpha}{2} h^2 \right) \\
&= \sup_{a \geq 0, D\geq 0} 
 \inf_{\zeta\in\fop_D, h\in\RR_+,\gamma\in\RR}  
 \left( P_{a}(\zeta, h,\gamma) - \gamma D - \frac{\beta\alpha}{2} h^2 \right) 
\end{align*}
which is independent of $n$, and we obtain by the previous lemma that  
\begin{equation}
\label{lisupFA} 
\limsup_n F^A_n \leq 
  \sup_{a \geq 0, D\geq 0} \inf_{\zeta\in\fop_D, h\in\RR_+,\gamma\in\RR}  
 \left( P_{a}(\zeta, h,\gamma) - \gamma D - \frac{\beta\alpha}{2} h^2 
  \right)  .
\end{equation} 
Using Lemma~\ref{F<F^A}, we obtain Proposition~\ref{up-bnd}.

\begin{proof}[Proof of Lemma~\ref{guerra}] 
Let $\zeta\in\fop_D$ be defined as in~\eqref{fop} with the parameters 
in~\eqref{lq}. The basic object in the proof of this lemma is the interpolated 
Hamiltonian $V_t : \RR_+^n \times \NN^K \to \RR$ defined for $t\in[0,1]$ as
\begin{equation}
\label{Vt} 
V_t(x,\bi) = \sqrt{t} \frac{\beta\kappa}{2\sqrt{n}} x^\T W x    
 + t \frac{\beta\alpha}{2} \frac{\ps{x}{1}^2}{n} 
 + \sqrt{1-t} \ps{x}{\bs q_{\bi}} 
 + \sqrt{t} \ps{1}{\bs y_{\bi}} 
  + (1-t) \beta\alpha h \ps{1}{x} 
\end{equation} 
along with a RPC $(v_{\bi})_{\bi \in\NN^K} \sim \RPC_{\bs\lambda}$. 
Here, the $\RR^n$--valued random vectors 
$\bs q_{\bi} = \begin{bmatrix} q_{\bi,l} \end{bmatrix}_{l=1}^n$ and 
$\bs y_{\bi} = \begin{bmatrix} y_{\bi,l} \end{bmatrix}_{l=1}^n$ are 
constructed as follows: the $n$ random processes 
$(q_{\bi,l})_{\bi\in\NN^K}$ for $l\in[n]$ are $n$ independent copies of the the 
process $(q_{\bi})_{\bi\in\NN^K}$ in~\eqref{qi}, the $n$ random processes 
$(y_{\bi,l})_{\bi\in\NN^K}$ are $n$ independent copies of the the 
process $(y_{\bi})_{\bi\in\NN^K}$ in~\eqref{yi}, and 
$(\bs q_{\bi})$, $(\bs y_{\bi})$, $(v_{\bi})$ and the matrix $W$ are 
independent.  

The standard way of establishing Lemma \ref{guerra} with the help of 
the function $V_t(x, \bi)$ and the RPC $(v_{\bi})$ is to set 
\[
\varphi(t) = \frac 1n \EE \log \sum_{\bi \in \NN^K} 
 v_{\bi} \int_{B_+^n(\sqrt{n} A)} 
  \exp V_t(x,\bi)  \ \mu_\beta^{\otimes n}(dx) , 
\]
to relate $\varphi(0)$ with the Parisi functional and $\varphi(1)$ with the free
energy $F^A$, and to control $\partial_t \varphi(t)$ with the help of the
well-known Gaussian Integration by Parts (IP) formula detailed in, \emph{e.g.},
\cite[Lemma 1.1]{pan-livre13} or \cite[\S1.3]{tal-livre11-t1}.  In the
classical contexts of the SK model or the so-called spherical model
\cite{tal-(sph)-06}, it holds by construction that the overlap $R_{11}$ is
equal to $1$ since the replicas live on the sphere with radius $\sqrt{n}$. In
this case, controlling $\partial_t \varphi(t)$ with the help of the Gaussian IP
formula is a simple matter of book keeping. This property of $R_{11}$ is
however not satisfied in our model, and this creates a difficulty which can be
circumvented by constaining the replicas to lie on thin spherical shells.  This
idea was implemented in \cite{pan-05}, building on the proof developed in
\cite{tal-06}.  The paper~\cite{pan-18} that we follow here also exploits this
idea. 

For $D \in (0, A^2)$ and $\varepsilon > 0$ small, denote as 
$\Delta^\varepsilon_n(D) \subset \RR_+^n$ the set 
\[
\Delta^\varepsilon_n(D)= \left\{ x \in \RR_+^n \ : \ 
   D-\varepsilon < \|x\|^2 / n < D+\varepsilon \right\} 
\]
Associate with this set the free energy 
\[
F_n^{\Delta^\varepsilon(D)}  = \frac 1n \EE\log 
 \int_{\Delta^\varepsilon(D)} \exp\left(H_n(x)\right) \, 
   \mu_\beta^{\otimes n}(dx)  
\]
Consider the Gibbs probability measure on the space 
$\Delta^\varepsilon(D) \times \NN^K$ defined as 
\[
\Gamma^{\Delta^\varepsilon(D)}_t(dx, \bi) \ \sim \ 
 v_{\bi} \exp V_t(x,\bi) \ \mu_\beta^{\otimes n}(dx) , 
\]
and let $\EGibbs{\cdot}_t$ be the expectation 
w.r.t.~$(\Gamma^{\Delta^\varepsilon(D)}_t)^{\otimes\infty}$. Define the 
function 
\[
\varphi^{\Delta^\varepsilon(D)}(t) = \frac 1n \EE \log \sum_{\bi \in \NN^K} 
 v_{\bi} \int_{\Delta^\varepsilon(D)} 
  \exp V_t(x,\bi) \ \mu_\beta^{\otimes n}(dx).  
\]
We have 
\begin{align*} 
\varphi^{\Delta^\varepsilon(D)}(1) &= F_n^{\Delta^\varepsilon(D)} + 
 \frac 1n \EE \log \sum_{\bi \in \NN^K} 
    v_{\bi} \exp ( \ps{1}{\bs y_{\bi}} )  
 = F_n^{\Delta^\varepsilon(D)} + \frac 12 
   \sum_{k=0}^{K-1} \lambda_k (\theta(b_{k+1}) - \theta(b_{k}))  \\ 
 &= F_n^{\Delta^\varepsilon(D)} 
 - \frac 12 \int \theta d\zeta + \frac{\theta(D)}{2} 
\end{align*} 
by the development that follows Proposition~\ref{P-RPC}. We also have 
\[
\varphi^{\Delta^\varepsilon(D)}(0) = 
 \frac 1n \EE \log \sum_{\bi \in \NN^K} 
 v_{\bi} \int_{\Delta^\varepsilon(D)} 
  \exp \left( \ps{x}{\bs q_{\bi}} + \beta\alpha h \ps{1}{x} \right) 
   \mu_\beta^{\otimes n}(dx)  
\]
On the set $\Delta^\varepsilon(D)$, it holds that 
\[
\forall\gamma \in \RR, \quad \gamma(nD - \| x \|^2) \leq 
   n |\gamma| \varepsilon, 
\]
therefore, for an arbitrary $\gamma\in\RR$, it holds that
\begin{align*} 
\varphi^{\Delta^\varepsilon(D)}(0) &\leq 
 \frac 1n \EE \log \sum_{\bi \in \NN^K} 
 v_{\bi} \int_{\Delta^\varepsilon(D)} 
  \exp \left( \ps{x}{\bs q_{\bi}} + \beta\alpha h \ps{1}{x} 
  + n |\gamma| \varepsilon - \gamma n D + \gamma \|x \|^2 \right) 
   \mu_\beta^{\otimes n}(dx)  \\
&\leq - \gamma D + |\gamma| \varepsilon + 
 \frac 1n \EE \log \sum_{\bi \in \NN^K} 
 v_{\bi} \int_{B_+(A\sqrt{n})} 
  \exp \left( \ps{x}{\bs q_{\bi}} + \beta\alpha h \ps{1}{x} 
   + \gamma \|x \|^2 \right) \mu_\beta^{\otimes n}(dx)  \\
&= - \gamma D + \frac 1n \bX_{0,\infty}^A(\zeta,h,\gamma) + |\gamma| \varepsilon . 
\end{align*}

To establish Guerra's bound, we compute the derivative 
$\partial_t \varphi^{\Delta^\varepsilon(D)}(t)$ by applying the Gaussian IP 
formula. We first observe that 
\begin{align*} 
\partial_t \varphi^{\Delta^\varepsilon(D)}(t) &= 
 \frac 1n \EE \EGibbs{\partial_t V_t(x,\bi)}_t  \\ 
 &= \frac 1n \EE \EGibbs{
 \frac{1}{2\sqrt{t}} \frac{\beta\kappa}{2\sqrt{n}} x^\T W x 
 - \frac{1}{2\sqrt{1-t}} \ps{x}{\bs q_{\bi}} 
 + \frac{1}{2\sqrt{t}} \ps{1}{\bs y_{\bi}} }_t 
+ \frac{\beta\alpha}{2}  
  \EE\EGibbs{\frac{\ps{x}{1}^2}{n^2} - 2 h \frac{\ps{1}{x}}{n} }_t . 
\end{align*} 
We apply the Gaussian IP formula to compute the first term at the right hand 
side of the last display. To that end, we need to compute 
\begin{align*}
U( (x,\bi), (y, \bj)) &= \EE\left[
\left( 
 \frac{1}{2\sqrt{t}} \frac{\beta\kappa}{2\sqrt{n}} x^\T W x 
 - \frac{1}{2\sqrt{1-t}} \ps{x}{\bs q_{\bi}} 
 + \frac{1}{2\sqrt{t}} \ps{1}{\bs y_{\bi}} 
\right) \times \right. \\
& 
  \quad \quad \quad \quad \quad \quad \quad \quad 
  \quad \quad \quad \quad \quad \quad \quad \quad 
\left. \left( 
  \sqrt{t} \frac{\beta\kappa}{2\sqrt{n}} y^\T W y 
 + \sqrt{1-t} \ps{y}{\bs q_{\bj}} 
 + \sqrt{t} \ps{1}{\bs y_{\bj}} 
\right) \right]  \\
&= \frac n2 \left( \xi\left( \ps{x}{y}/n \right) - 
 \left( \ps{x}{y}/n \right) \xi'(b_{\bi \wedge \bj}) 
  + \theta(b_{\bi \wedge \bj}) \right) , 
\end{align*}
where, writing $\bi = (i_1,\ldots i_K)$ and $\bj = (j_1,\ldots, j_K)$, we 
set $\bi \wedge \bj = \max \{ l  :  (i_1,\ldots,i_l) = (j_1,\ldots, j_l) \}$.
Remembering that $b_K = D$ since our Parisi measure belongs to $\fop_D$, we 
then have 
\begin{align}
\partial_t \varphi^{\Delta^\varepsilon(D)}(t) &= 
\frac 12 \EE\EGibbs{\xi(R_{11}) -  R_{11} \xi'(D) + \theta(D)}_t 
- \frac 12 \EE\EGibbs{\xi(R_{12}) -  R_{12} \xi'(b_{\bi \wedge \bj}) 
  + \theta(b_{\bi \wedge \bj})}_t  \nonumber \\
&\phantom{=} 
+ \frac{\beta\alpha}{2}  \EE\EGibbs{\left( \ps{x}{1}/n - h \right)^2}_t  
 - \frac{\beta\alpha}{2} h^2  \nonumber 
\end{align} 
For our function $\xi$, it holds that 
$\xi(x) - x \xi'(y) + \xi(y) = 0.5 \beta^2 \kappa^2 (x - y)^2$. Therefore, 
remembering that $|R_{11} - D| < \varepsilon$ and that $\alpha\leq 0$, we 
obtain that 
\begin{align}
\label{phi'} 
\partial_t \varphi^{\Delta^\varepsilon(D)}(t) &= 
\frac{\beta^2\kappa^2}{4} \EE\EGibbs{(R_{11} - D)^2}_t 
- \frac{\beta^2\kappa^2}{4} 
  \EE\EGibbs{(R_{12} - b_{\bi \wedge \bj})^2}_t 
+ \frac{\beta\alpha}{2}  \EE\EGibbs{\left( \ps{x}{1}/n - h \right)^2}_t  
- \frac{\beta\alpha}{2} h^2 \\
&\leq \frac{\beta^2\kappa^2}{4} \varepsilon^2 - \frac{\beta\alpha}{2} h^2. 
\nonumber 
\end{align} 
We therefore have that 
$\varphi^{\Delta^\varepsilon(D)}(1) \leq \varphi^{\Delta^\varepsilon(D)}(0)
 + \frac{\beta^2\kappa^2}{4} \varepsilon^2 - \frac{\beta\alpha}{2} h^2$, 
which implies that 
\begin{align*}
F_n^{\Delta^\varepsilon(D)} 
 &\leq -\gamma D + \frac 1n \bX_{0,\infty}^A(\zeta,h,\gamma) 
  - \frac{\beta\alpha}{2} h^2 
 + \frac 12 \int \theta d\zeta - \frac{\theta(D)}{2} 
+ \frac{\beta^2\kappa^2}{4} \varepsilon^2 + |\gamma|\varepsilon  \\ 
 &= \bP_{\infty}^A(\zeta,h,\gamma) - \gamma D - \frac{\beta\alpha}{2} h^2 
+ \frac{\beta^2\kappa^2}{4} \varepsilon^2 + |\gamma|\varepsilon 
\end{align*} 
for arbitrary $\zeta\in\fop_D$, $h\geq 0$ and $\gamma$. 
By the argument of \cite[Lemma~3]{pan-18}, this implies Lemma~\ref{guerra}. 

\end{proof}

\section{Proof of Theorem~\ref{F->P}: lower bound for $\alpha\leq 0$} 
\label{sec:lbnd} 

The purpose of this section is to establish the following proposition:
\begin{proposition}
\label{low-bnd} 
It holds that 
\[
\liminf_n \widetilde F_n \geq 
  \sup_{a \geq 0, D\geq 0} \inf_{\zeta\in\fop_D, h\geq 0,\gamma\in\RR} 
 \left( P_a(\zeta, h,\gamma) - \gamma D - \frac{\beta\alpha}{2} h^2 \right) .
\]
\end{proposition} 
This bound holds true whatever is the value of $\alpha$. However, it is not 
tight in general for $\alpha > 0$. A tight lower bound for $\alpha > 0$ 
will be provided in Section~\ref{low-alpha>0} below. 

To prove this proposition, we shall lower bound $\liminf_n F_{a,n}$ for $a > 0$
and use Lemma~\ref{F>Fa}.  For $D \in (0, a^2)$ and $\varepsilon > 0$ small, we
still define the set $\Delta^\varepsilon_n(D) \subset \RR_+^n$ as above, and we
circumvent the domain of integration for $F_{a,n}$, working with
$F_{a,n}^{\Delta^\varepsilon_n(D)}$ defined as 
\[
F_{a,n}^{\Delta^\varepsilon_n(D)} = 
  \frac 1n \EE\log \cZ_{a,n}^{\Delta^\varepsilon_n(D)} \quad \text{with}\quad 
 \cZ_{a,n}^{\Delta^\varepsilon_n(D)} = 
  \int_{\Delta^\varepsilon_n(D) \cap [0,a]^n} 
  e^{H_n(x)} \mu_\beta^{\otimes n}(dx) . 
\] 

The proof relies on the so-called Aizenman-Sims-Starr (ASS) scheme, which goes
as follows in our context.  The equations that we will write right away will
just serve to set the stage; we shall need to modify them afterwards. 

Fixing an integer $m > 0$, we partition a vector $v \in \RR_+^{n+m}$ as 
$v = ( u, x )$ with $u \in \RR_+^m$ and $x \in \RR_+^n$. Since 
\[
\Delta^\varepsilon_{n+m}(D) \supset \left\{
 v = (u,x) \in \RR_+^n \, : \, u \in \Delta^\varepsilon_{m}(D),  
 x \in \Delta^\varepsilon_{n}(D)  \right\}, 
\]
we can write 
\begin{align}
\liminf_n F_{a,n}^{\Delta^\varepsilon_n(D)} &\geq 
 \liminf_n \frac 1m \left( 
 \EE\log \cZ_{a,n+m}^{\Delta^\varepsilon_{n+m}(D)} 
 - \EE\log \cZ_{a,n}^{\Delta^\varepsilon_n(D)} \right)  \nonumber   \\
&\geq 
 \liminf_n \frac 1m \left( 
 \EE\log \int_{\Delta^\varepsilon_m(D) \cap [0,a]^m} \mu_\beta^{\otimes m}(du)  
  \int_{\Delta^\varepsilon_n(D) \cap [0,a]^n}  \mu_\beta^{\otimes n}(dx)  
  \ e^{H_{n+m}(u,x)} \right. \nonumber\\
&\phantom{=} \quad\quad\quad\quad\quad \left. 
 - \EE\log \int_{\Delta^\varepsilon_n(D) \cap [0,a]^n} 
  e^{H_{n}(x)}  \mu_\beta^{\otimes n}(dx)  \right) \nonumber \\
&\eqdef \liminf_n \chi_{m,n} . 
\label{F->chi} 
\end{align} 
Adapting a standard derivation to our case, see \cite[\S 1.3]{pan-livre13},
with the small particularity that we need now to consider the term 
$(\beta\alpha/2) \ps{v}{1}^2 / (n+m)$ in the expression of $H_{n+m}(v)$, we 
obtain 
\begin{align} 
\chi_{m,n} &= \bs\chi_{m,n} + o_n(1), \quad \text{with} \nonumber \\
\bs\chi_{m,n} &= \frac 1m  
\EE \log \EGibbs{\int_{\Delta^\varepsilon_m(D) \cap [0,a]^m}  
 \exp\left( 
 \ps{u}{Q(x)} + \beta\alpha \ps{u}{1_m} \frac{\ps{x}{1_n}}{n}  
 \right) \mu_\beta^{\otimes m}(du)}' \nonumber  \\
 &\phantom{=} 
  - \frac 1m \EE \log \EGibbs{ \exp\left( \sqrt{m} Y(x) 
    + m \frac{\beta\alpha}{2} \frac{\ps{x}{1_n}^2}{n^2} \right) }', 
\label{chimn} 
\end{align} 
with the following notations: the expectation $\EGibbs{\cdot}'$ is taken 
with respect to $(G_n')^{\otimes\infty}$, where 
$G_n'(dx) \in \cP(\Delta^\varepsilon_n(D) \cap [0,a]^n)$ is the random 
probability measure defined as 
\begin{equation*}
G'_n(dx) \sim \exp\left( H'_n(x) \right) \mu_\beta^{\otimes n}(dx) ,  
\end{equation*} 
with $H'_n(x)$ being the Hamiltonian 
\[
H'_n(x) = \frac{\beta\kappa}{2\sqrt{n+m}} x^\T W_n x + 
         \frac{\beta\alpha}{2} \frac{\ps{x}{1_n}^2}{n+m} , 
\] 
and $Q : \RR_+^n \to \RR^m$ and $Y : \RR_+^n \to \RR$ are two Gaussian centered
processes, independent of $W_n$, and which probability distributions are 
defined through the matrix and scalar covariances 
\begin{align*}
\EE Q(x^1) Q(x^2)^\T &= \beta^2\kappa^2 R_{12} I_m = \xi'(R_{12}) I_m,  
    \ \text{and} \\
\EE Y(x^1) Y(x^2) &= \frac{\beta^2\kappa^2}{2} R_{12}^2 = \theta(R_{12}), 
\end{align*} 
with $R_{12} = \ps{x^1}{x^2} / n$. 

By adapting the proof of \cite[Theorem~1.3]{pan-livre13}, see also \cite[Page
878]{pan-18}, we obtain that for each fixed $m > 0$, $\bs\chi_{m,n}$ is a
continuous functional of the distribution of the couple
$\left((R_{i,j})_{i,j\geq 1}, \left( \ps{x^k}{1} / n\right)_{k\geq 1} \right)$
of the infinite array of overlaps $(R_{i,j})_{i,j\geq 1}$ and the infinite
vector $( \ps{x^k}{1} / n)_{k\geq 1}$ under the distribution $\EE
(G_n')^{\otimes\infty}$. Here the continuity is with the respect to the
topology of the narrow convergence of the finite dimensional distributions of
$\left((R_{i,j})_{i,j\geq 1}, ( \ps{x^k}{1} / n)_{k\geq 1} \right)$. 

Leaving aside the vector $( \ps{x^k}{1} / n)_{k\geq 1}$ for a moment, the
principle of the proof for the limit inferior of the free energy stands as
follows. It is usually required that the distribution of the array of
overlaps $(R_{i,j})_{i,j\geq 1}$ satisfies the celebrated Ghirlanda-Guerra (GG)
identities in the large--$n$ limit, so that in this limit, these overlaps can
be seen as issued from replicas sampled from a Gibbs measure on a Hilbert space
described in terms of a Ruelle probability cascade.  Applying one of the
important ideas in spin glass theory, the GG identities can be obtained by
properly perturbing the Hamiltonian of the measure $G'_n$ without much
affecting the free energy \cite[Chapter~12]{tal-livre11-t2},
\cite[Chapter~3]{pan-livre13}. In our context, this should be complemented with
another idea, dating back to \cite{pan-18}: since our replicas live in a
thickening $\Delta^\varepsilon_n(D)$ of a sphere, and not exactly on this
sphere, a transformation of these replicas is needed before applying the
perturbation on the Hamiltonian in order to obtain the GG identities for large
$n$. It will be the array of overlaps of the transformed replicas, denoted as 
$(\widetilde R_{ij})$ below, that will satisfy the GG identities in the 
large--$n$ limit. 

In our specific model, we also need to manage the vector of empirical means $(
\ps{x^k}{1} / n)_{k\geq 1}$, requiring these empirical means to concentrate in
the large--$n$ limit. To that end, we shall add a supplementary perturbation to
the Hamiltonian of $G'_n$. This is the main specificity of our proof as 
regards the lower bound on $F_{a,n}$. 

To implement these ideas, we now resume our argument from the beginning by
perturbing our Hamiltonian $H_n(x)$. Keeping our $\varepsilon > 0$, let
$D_\varepsilon = D \1_{D\geq\sqrt{\varepsilon}}$. For $x \in
\Delta_n^\varepsilon(D)$, define the function $x\mapsto\tx$ as 
\begin{equation}
\label{xtx} 
\tx = \left\{\begin{array}{cl} 
 \displaystyle{\sqrt{\frac{D}{\| x\|^2 /n}}} x 
         &\text{if } D\geq\sqrt{\varepsilon}  \\ 
  0 &\text{if not.} 
 \end{array}
\right. , 
\end{equation} 
in such a way that $\tx$ lives on the sphere of radius $\sqrt{D_\varepsilon}$.
For a given $n > 0$, let $(g_{n,j})_{j\geq 1}$ is a sequence of scalar 
independent centered Gaussian processes on $\RR_+^n$ such that 
$\EE g_{n,j}(x^1) g_{n,j}(x^2) = (\ps{x^1}{x^2}/n)^j$,
see \cite[\S 3.2]{pan-livre13} for the construction of such processes. Writing 
\[
\eta_n(x) = \frac{\ps{x}{1}}{n}, 
\]
our perturbed version $\Hpert_n$ of the Hamiltonian $H_n$ will take the form 
\[
\Hpert_n(x) = 
 H_n(x) + n^\varrho \sum_{j\geq 1} (2a)^{-j} w_j g_{n,j}(\tx) 
   + n^\delta s \eta_n(x), 
\]
where $\varrho, \delta > 0$, the elements of the sequence $(w_j)_{j\geq 1}$
take their values in the interval $[0,3]$, and $s \in[0,3]$.  
Let 
\begin{align*} 
G_{a,n}^{\text{pert},\Delta^\varepsilon_n(D)}(dx) &= 
  \frac{e^{\Hpert_n(x)}}{\cZ_{a,n}^{\text{pert},\Delta^\varepsilon_n(D)}} 
    \mu_\beta^{\otimes n}(dx)  
 \in \cP(\Delta^\varepsilon_n(D) \cap [0,a]^n), \quad \text{with} \\ 
 \cZ_{a,n}^{\text{pert},\Delta^\varepsilon_n(D)} &= 
  \int_{\Delta^\varepsilon_n(D) \cap [0,a]^n} 
  e^{\Hpert_n(x)} \mu_\beta^{\otimes n}(dx) 
\end{align*} 
be the Gibbs measure constructed from this Hamiltonian, and let 
$\EGibbs{\cdot}$ be the mean with respect to 
$(G_{a,n}^{\text{pert},\Delta^\varepsilon_n(D)})^{\otimes\infty}$. 
We also define the overlaps $\widetilde R_{ij} = \ps{\tx^i}{\tx^j}/n$, where
$x^i$ and $x^j$ are two replicas under 
$(G_{a,n}^{\text{pert},\Delta^\varepsilon_n(D)})^{\otimes\infty}$, and where 
$\tx^i$ and $\tx^j$ are the transformations of $x^i$ and $x^j$ by \eqref{xtx}
respectively. 

\begin{lemma} 
\label{GGas} 
Assume that $\varrho \in (1/4, 1/2)$ and $\delta \in (1/2,1)$. Then, the free 
energy $F_{a,n}^{\text{pert},\Delta^\varepsilon_n(D)} = n^{-1} \EE\log 
\cZ_{a,n}^{\text{pert},\Delta^\varepsilon_n(D)}$ satisfies 
\begin{equation} 
\label{Fpert} 
F_{a,n}^{\Delta^\varepsilon_n(D)} - 
F_{a,n}^{\text{pert},\Delta^\varepsilon_n(D)} \xrightarrow[n\to\infty]{} 0. 
\end{equation} 
Assume now that $(s, w_1, w_2, \ldots)$ is a sequence of i.i.d.~random
variables distributed uniformly on the interval $[1,2]$ and independent of
$W_n$, and denote as $\EE_{s,w}$ the expectation with respect to this sequence.
For each integers $k\geq 2$ and $p \geq 1$ and each bounded measurable
function $f = f((\widetilde R_{ij})_{1\leq i,j\leq k})$ of the overlaps
$(\widetilde R_{ij})_{1\leq i,j\leq k}$, it holds that 
\[
\EE_{s,w} \left| 
 \EE\EGibbs{f \widetilde R^p_{1,k+2}} 
  - \frac 1k \EE\EGibbs{f} \EE\EGibbs{\widetilde R^p_{12}} 
  - \frac 1k \sum_{i=2}^k \EE\EGibbs{f \widetilde R^p_{1,i}} 
\right| \xrightarrow[n\to\infty]{} 0. 
\]
Finally, writing $\eta = \ps{x}{1}/n$ with $x$ being distributed as 
$G_{a,n}^{\text{pert},\Delta^\varepsilon_n(D)}$, it holds that  
\[
\EE_{s,w} \EE\EGibbs{\left| \eta - \EE\EGibbs{\eta} \right|} 
       \xrightarrow[n\to\infty]{} 0. 
\]
\end{lemma} 
\begin{proof} 
Writing $g(\tx) = \sum_{j\geq 1} (2a)^{-j} w_j g_{n,j}(\tx)$, we have by 
Jensen's inequality 
\[
\EE \log \frac{\int \exp(H + n^\varrho g) d\mu_\beta^{\otimes n}}
     {\int \exp(H) d\mu_\beta^{\otimes n}}  
\geq \EE \frac{\int n^\varrho g \exp(H) d\mu_\beta^{\otimes n}}
   {\int \exp(H) d\mu_\beta^{\otimes n}} = 0, 
\]
therefore, 
\[
F_{a,n}^{\Delta^\varepsilon_n(D)} 
 = \frac 1n \EE\log \cZ_{a,n}^{\Delta^\varepsilon_n(D)}  
 \leq \frac 1n \EE\log \int\exp(H + n^\varrho g) d\mu_\beta^{\otimes n}
\leq \frac 1n \EE\log \int \exp(H + n^\varrho g + n^\delta s \eta) 
      d\mu_\beta^{\otimes n}
= F_{a,n}^{\text{pert},\Delta^\varepsilon_n(D)}. 
\]
By Jensen's inequality involving this time the expectation with respect to the
law of the process $g$, we also have 
\begin{align*} 
F_{a,n}^{\text{pert},\Delta^\varepsilon_n(D)} &\leq 
3 a n^{\delta - 1} + 
 \frac 1n \EE\log \int \exp(H + n^\varrho g) d\mu_\beta^{\otimes n}
\leq 3 a n^{\delta - 1} + 
 \frac 1n \EE\log \int \exp(H) \exp(n^{2\varrho} \EE g^2 / 2) 
      d\mu_\beta^{\otimes n}  \\
&\leq 3 a n^{\delta - 1} + 1.5 n^{2\varrho - 1}  +  
  F_{a,n}^{\Delta^\varepsilon_n(D)} ,  
\end{align*} 
hence the convergence~\eqref{Fpert}. 

The second convergence result is obtained by a straightforward adaptation of
the proof of \cite[Theorem~3.2]{pan-livre13} towards dealing with the overlaps
$\widetilde R_{ij}$. Here, the replacement of $x$ with $\tx$ in the expression
of $g(\tx)$ plays an important role. 

To establish the last convergence, we actually follow the same canvas as 
for the proof of \cite[Theorem~3.2]{pan-livre13}. Recall the presence of the 
term $n^\delta s \eta_n$ in the expression of $\Hpert_n$ above, and define  
the function 
$s\mapsto\varphi(s) = \log \cZ_{a,n}^{\text{pert},\Delta^\varepsilon_n(D)}$. 
It is clear that 
$\varphi'(s) = n^\delta \EGibbs{\eta}$ and 
$\varphi''(s) = n^{2\delta} \left( \EGibbs{\eta^2} - \EGibbs{\eta}^2\right) 
 \geq 0$. With this at hand, we have 
\[
n^{2\delta} \int_1^2 \EE\EGibbs{ (\eta - \EGibbs{\eta})^2} ds = 
\EE\int_0^1 \varphi''(s) ds = \EE \varphi'(2) - \EE \varphi'(1) 
 \leq \EE \varphi'(2) \leq n^\delta a. 
\]
It therefore holds that 
\[
\int_1^2 \EE\EGibbs{ |\eta - \EGibbs{\eta}|} ds 
   \leq \frac{\sqrt{a}}{n^{\delta/2}} . 
\]
We now bound 
\[
\EE| \EGibbs{\eta} - \EE\EGibbs{\eta} | = 
  \frac{\EE|\varphi'(s) - \EE\varphi'(s)|}{n^\delta}.  
\]
Here, we shall need the quantitative version of the so-called Griffith lemma, 
given by~\cite[Lemma~3.2]{pan-livre13}. Namely, if 
$f,\bs f:\RR\to\RR$ are two convex differentiable functions, then,
for any $\epsilon > 0$, it holds that
\[
|f'(s) - \bs f'(s) | \leq 
 \bs f'(s+\epsilon) - \bs f'(s-\epsilon) + 
  \frac{|f(s+\epsilon) - \bs f(s+\epsilon)| 
   + |f(s-\epsilon) - \bs f(s-\epsilon)| + |f(s) - \bs f(s)|}{\epsilon}. 
\]
We shall use this result with $f = \varphi$ and $\bs f = \EE\varphi$. 
Observing that $\EE (H(x) + n^\varrho g(\tx))^2\leq n R_{11} + 3 n^{2\varrho}$, 
we obtain by Gaussian concentration \cite[Theorem~1.2]{pan-livre13} that 
\[
\sup \{ \EE|\varphi(v) - \EE\varphi(v)| \ : \ s, w_1, w_2, \ldots \in [0,3] \} 
 \leq C \sqrt{n}. 
\]
Therefore, since $\varphi$ is $a n^\delta$--Lipschitz, we have 
\begin{align*} 
 \int_1^2 \EE| \EGibbs{\eta} - \EE\EGibbs{\eta} | ds &= 
 n^{-\delta} \int_1^2 \EE|\varphi'(s) - \EE\varphi'(s)| ds \\
 &\leq n^{-\delta} \int_1^2 \EE\varphi'(s+\epsilon) ds 
 - \int_1^2 \EE\varphi'(s-\epsilon) ds 
  + C \frac{\sqrt{n}}{\epsilon n^\delta}  \\
 &= \frac{\EE\varphi(2+\epsilon) - \EE\varphi(2-\epsilon) 
 - (\EE\varphi(1+\epsilon) - \EE\varphi(1-\epsilon))}{n^\delta}  
  + C \frac{\sqrt{n}}{\epsilon n^\delta} \\
&\leq 2 a \epsilon + C \frac{\sqrt{n}}{\epsilon n^\delta}. 
\end{align*} 
Taking $\epsilon = n^{1/4 - \delta/2}$, we obtain our last convergence. 
\end{proof}

We now follow the approach of \cite{pan-18} by pointing out the specificity of
our model related with the presence of the empirical means of the replicas.
Take $\varrho \in (1/4, 1/2)$ and $\delta \in (1/2,1)$. Then, the
convergence~\eqref{Fpert} holds true by the previous lemma.  Recall the 
expression of $H'_n(x)$, let $G_{a,n}^{\text{pert}',\Delta^\varepsilon_n(D)}$ 
be the Gibbs measure constructed from the Hamiltonian 
\[
H^{\text{pert}'}_n(x) =  
 H_n'(x) + n^\varrho \sum_{j\geq 1} (2a)^{-j} w_j g_{n,j}(\tx) 
   + n^\delta s \eta_n(x), 
\]
and denote as $\EGibbs{\cdot}'$ the expectation with respect to 
$(G_{a,n}^{\text{pert}',\Delta^\varepsilon_n(D)})^{\otimes\infty}$.  
Define the overlaps $\widetilde R_{ij} = \ps{\tx^i}{\tx^j}/n$, where
$x^i$ and $x^j$ are two replicas under 
$(G_{a,n}^{\text{pert}',\Delta^\varepsilon_n(D)})^{\otimes\infty}$, and where 
$\tx^i$ and $\tx^j$ are the transformations of $x^i$ and $x^j$ by \eqref{xtx}
respectively. 
For each integers $k\geq 2$ and $p \geq 1$ and each bounded measurable
function $f = f((\widetilde R_{ij})_{1\leq i,j\leq k})$, write 
\[
E(f, k, p) = \EE\EGibbs{f \widetilde R^p_{1,k+2}} 
  - \frac 1k \EE\EGibbs{f} \EE\EGibbs{\widetilde R^p_{12}} 
  - \frac 1k \sum_{i=2}^k \EE\EGibbs{f \widetilde R^p_{1,i}} . 
\]
Then, by a slight modification of the proof of \cite[Lemma 3.3]{pan-livre13}
based on Lemma~\ref{GGas} above, we can show that for each $n$, there exists a
deterministic sequence $(s^n, w_1^n, w_2^n, \ldots)$ entering the construction
of $H^{\text{pert}'}_n$, and such that 
\begin{equation}
\label{GG+h} 
\begin{aligned}
& E(f, k, p) \xrightarrow[n\to\infty]{} 0 \quad 
 \text{for each } k \geq 2, p\geq 1, \text{ and monomial } 
  f = f((\widetilde R_{ij})_{1\leq i,j\leq k}) , \ \text{and} \\ 
&\EE\EGibbs{\left| \eta - \EE\EGibbs{\eta} \right|} 
  \xrightarrow[n\to\infty]{} 0. 
\end{aligned}
\end{equation}  
Furthermore, similarly to what we obtained in Equations~\eqref{F->chi} 
and~\eqref{chimn} above through the ASS scheme, it holds that 
\begin{equation}
\label{liminfF} 
\liminf_n F_{a,n}^{\text{pert},\Delta^\varepsilon_n(D)} \geq \liminf_n
\bs\chi_{m,n}^{\text{pert}}, 
\end{equation} 
where $\bs\chi_{m,n}^{\text{pert}}$ has the same expression as $\bs\chi_{m,n}$ 
in~\eqref{chimn} with the difference that $\EGibbs{\cdot}'$ is
now the expectation with respect to 
$(G_{a,n}^{\text{pert}',\Delta^\varepsilon_n(D)})^{\otimes\infty}$. In the
remainder, the sequence $(s^n, w_1^n, w_2^n, \ldots)$ is chosen for each $n$
in such a way that all these properties are satisfied. 

As a next step, we need to replace the processes $Q(x)$ and $Y(x)$ in the
expression of $\bs\chi_{m,n}^{\text{pert}}$ with $Q(\tx)$ and $Y(\tx)$
respectively. Writing 
\begin{align*} 
\bs{\tilde\chi}_{m,n}^{\text{pert}}
 &=  
 \frac 1m \EE \log \EGibbs{\int_{\Delta^\varepsilon_m(D) \cap [0,a]^m}  
 \exp\left( 
 \ps{u}{Q(\tx)} + \beta\alpha \ps{u}{1_m} \frac{\ps{x}{1_n}}{n}  
 \right) \mu_\beta^{\otimes m}(du)}' \nonumber  \\
 &\phantom{=} 
  - \frac 1m \EE \log \EGibbs{ \exp\left( \sqrt{m} Y(\tx) 
    + m \frac{\beta\alpha}{2} \frac{\ps{x}{1_n}^2}{n^2} \right) }', 
\end{align*}
it is true that 
\begin{equation} 
\label{chi-chi} 
 \left| \bs{\tilde\chi}_{m,n}^{\text{pert}} - \bs\chi_{m,n}^{\text{pert}} 
  \right| \leq C \varepsilon^{1/4}, 
\end{equation} 
see \cite[\S 6]{pan-18}, which allows us to focus the lower bound analysis on 
$\bs{\tilde\chi}_{m,n}^{\text{pert}}$. This quantity is a continuous 
functional of the distribution of the couple 
$\left((\widetilde R_{i,j})_{i,j\geq 1}, 
  \left( \ps{x^k}{1} / n \right)_{k\geq 1} \right)$, under the distribution 
$\EE(G_{a,n}^{\text{pert}',\Delta^\varepsilon_n(D)})^{\otimes\infty}$, in the
topology of the narrow convergence of the finite dimensional distributions of
$\left((\widetilde R_{i,j})_{i,j\geq 1}, ( \ps{x^k}{1} / n)_{k\geq 1} \right)$. 

Let $\zeta \in \fop_D$ be defined as in~\eqref{lq} and~\eqref{fop}, and let
$(v_{\bi})_{\bi\in\NN^K} \sim \RPC_{\bs\lambda}$.  Consider the two Gaussian
processes $( q_{\bi} )_{\bi\in\NN^K}$ and $( y_{\bi} )_{\bi\in\NN^K}$
independent of $(v_{\bi})$ and such that 
\[
\EE q_{\bi^1} q_{\bi^2} = \xi'(b_{\bi^1\wedge\bi^2}) \quad \text{and} \quad  
\EE y_{\bi^1} y_{\bi^2} = \theta(b_{\bi^1\wedge\bi^2}). 
\]
Given an integer $m > 0$, let $(\bq_{\bi} )_{\bi\in\NN^K}$ be a $\RR^m$--valued
Gaussian process made of $m$ independent copies of $( q_{\bi} )_{\bi\in\NN^K}$.
Define the functions 
\begin{align*} 
f^1_m(\Delta^\varepsilon_m(D),a,\zeta,h) &= 
  \frac 1m \EE \log \sum_{\bi\in\NN^K} v_{\bi} 
 \int_{\Delta^\varepsilon_m(D) \cap [0,a]^m} 
 \exp\left( \ps{u}{\bq_{\bi}} + \beta\alpha h \ps{u}{1_m} \right) 
   \mu_\beta^{\otimes m}(du), \quad \text{and} \\ 
 f^2(\zeta) &= 
   \frac 1m \EE \log \sum_{\bi\in\NN^K} v_{\bi}  
  \exp\left(\sqrt{m} y_{\bi} \right) =   
 \frac 12 \sum_{k=0}^{K-1} \lambda_k (\theta(b_{k+1})- \theta(b_{k})) , 
\end{align*} 
We also define 
\[
\Phi_a(\zeta,h,\gamma) = 
  \EE \log \sum_{\bi\in\NN^K} v_{\bi} 
 \int_0^a 
 \exp\left( u q_{\bi} + \beta\alpha h u + \gamma u^2\right) \mu_\beta(du) . 
\]
We have the two following lemmas: 
\begin{lemma}
\label{pan-lm6} 
$\displaystyle{\liminf_m f^1_m(\Delta^\varepsilon_m(D),a,\zeta,h) 
 \geq \inf_{\gamma} \left( \Phi_a(\zeta,h,\gamma) - \gamma D \right)}$. 
\end{lemma}
This lemma is proven by a straightforward adaptation of the proof of 
\cite[Lemma 6]{pan-18}. 

Defining the distance $\bd$ on $\cP([0,D])$ as 
\begin{equation}
\label{defdist} 
\bd(\zeta,\nu) = \int_0^D \left| \zeta([0,t]) - \nu([0,t]) \right| dt, 
\end{equation} 
we also have 
\begin{lemma}
\label{f12lip} 
It holds that 
\[
\left| f^1_m(\Delta^\varepsilon_m(D),a,\zeta,h) - 
 f^1_m(\Delta^\varepsilon_m(D),a,\tilde \zeta,\tilde h) \right| 
 \leq C \left( \bd(\zeta,\tilde\zeta) + | h - \tilde h | \right). 
\]
for each measures $\zeta, \tilde\zeta \in \fop_D$ and reals 
$h,\tilde h \geq 0$, where $C > 0$ is independent of $m$. Moreover, 
\[
\left| f^2(\zeta) - f^2(\tilde \zeta) \right| \leq C \bd(\zeta,\tilde\zeta) .
\]
\end{lemma} 
The second bound is well-known. 
We sketch the proof of the first bound, which analogue in the SK literature
is also well-known, see \cite{gue-01} or \cite[Theorem~1]{gue-03}. Here we 
slightly modify the proof of \cite[Lemma 7]{pan-potts-18}. 
\begin{proof}[Sketch of the proof of the first bound] 
Without loss of generality, we may assume that
the measures $\zeta$ and $\tilde\zeta$ admit the same coefficients
$(\lambda_k)$, \emph{i.e.}, it can be shown that without modifying neither
$f^1_m(\Delta^\varepsilon_m(D),a,\zeta,h)$ nor
$f^1_m(\Delta^\varepsilon_m(D),a,\tilde \zeta,\tilde h)$ nor 
$\bd(\zeta,\tilde\zeta)$, one can assume that $\zeta$ and $\tilde\zeta$ take 
the following forms: there exists an integer $K > 0$ and real numbers 
\begin{align*} 
 & 0 < \lambda_0 < \cdots < \lambda_{K-1} < 1, \\ 
 & 0 = b_0 \leq \cdots \leq b_{K-1} \leq b_K = D, \quad \text{and} \\ 
 & 0 = \tilde b_0 \leq \cdots \leq \tilde b_{K-1} \leq \tilde b_K = D, 
\end{align*} 
such that $\zeta([0,t]) = \lambda_k$ if $t\in[b_k, b_{k+1})$ and such that
$\tilde\zeta([0,t]) = \lambda_k$ if $t\in[\tilde b_k, \tilde b_{k+1})$. See
\cite[Proposition 6.3]{dom-mou-livre24} for a justification.

Let $(\bq_{\bi})_{\bi\in\NN^K}$ be the $\RR^m$--valued Gaussian process
constructed from $\zeta$ as in the definition of
$f^1_m(\Delta^\varepsilon_m(D),a,\zeta,h)$ above, and let
$(\tilde\bq_{\bi})_{\bi\in\NN^K}$ a $\RR^m$--valued Gaussian process
independent of $(\bq_{\bi})$ and $(v_{\bi})$, and constructed from
$\tilde\zeta$ similarly to $(\bq_{\bi})$.  For $t \in [0,1]$, define the
process $(\bq_{\bi}(t))_{\bi\in\NN^K}$ as $\bq_{\bi}(t) = \sqrt{t} \bq_{\bi} +
\sqrt{1-t} \tilde\bq_{\bi}$. Also let $h(t) = t h + (1-t) \tilde h$, and let
$G_t(\bi, dx) \in \cP(\NN^K \times (\Delta^\varepsilon_m(D) \cap [0,a]^m))$ be
the Gibbs measure $G_t(\bi, du) \sim v_{\bi} \exp\left( \ps{u}{\bq_{\bi}(t)} +
\beta\alpha h(t) \ps{u}{1_m} \right) \mu_\beta^{\otimes m}(du)$ with mean
$\EGibbs{\cdot}_t$. Define 
\[
\varphi(t) = \frac 1m \EE \log \sum_{\bi\in\NN^K} v_{\bi} 
 \int_{\Delta^\varepsilon_m(D) \cap [0,a]^m} 
 \exp\left( \ps{u}{\bq_{\bi}(t)} + \beta\alpha h(t) \ps{u}{1_m} \right) 
   \mu_\beta^{\otimes m}(du), 
\]
in such a way that $f^1_m(\Delta^\varepsilon_m(D),\zeta,h) = \varphi(1)$ and 
$f^1_m(\Delta^\varepsilon_m(D),\tilde\zeta,\tilde h) = \varphi(0)$. We have
\[
\varphi'(t) = \frac{1}{2m} \EE\EGibbs{ \frac{\ps{u}{\bq_{\bi}}}{\sqrt{t}} 
  - \frac{\ps{u}{\tilde\bq_{\bi}}}{\sqrt{1-t}} } 
+ \frac{\beta\alpha}{m} (h - \tilde h) \EE\EGibbs{\ps{u}{1_m}} . 
\]
The first term is treated by the IP formula as in the proof 
of~\cite[Lemma 7]{pan-potts-18}, and leads to the $\bd(\zeta,\tilde\zeta)$ 
term in the statement. The second term is bounded by 
$a \beta |\alpha| | h - \tilde h |$. 
\end{proof} 
We can now finish the proof of Proposition~\ref{low-bnd}. Recalling the 
convergence \eqref{Fpert} and the bounds~\eqref{liminfF} and~\eqref{chi-chi},
we obtain that 
\[
\liminf_n F_{a,n}^{\Delta^\varepsilon_n(D)} = 
\liminf_n F_{a,n}^{\text{pert},\Delta^\varepsilon_n(D)} \geq 
\liminf_m \liminf_n \bs{\tilde\chi}_{m,n}^{\text{pert}} - C \varepsilon^{1/4}.
\]
We now apply the well-known theory detailed in, \emph{e.g.}, \cite[\S
3.6]{pan-livre13}, as regards the treatment of the $\liminf_n$ at the right
hand side of this inequality.  For a given $m$, consider a sub-sequence of
$(n)$ converging to infinity, along which
$\bs{\tilde\chi}_{m,n}^{\text{pert}}$ converges to its limit inferior in $n$,
and the couple 
$\left((\widetilde R_{i,j})_{i,j\geq 1}, ( \ps{x^k}{1} / n)_{k\geq 1} \right)$ 
converges in distribution under
$\EE(G_{a,n}^{\text{pert}',\Delta^\varepsilon_n(D)})^{\otimes\infty}$.  By the
convergences~\eqref{GG+h}, the limit distribution of the array $(\widetilde
R_{i,j})_{i,j\geq 1}$ satisfies the GG identities, and the replicas
$\ps{x^k}{1} / n$ converge in probability towards a deterministic number $h^{m}
\in [0, \sqrt{D}]$. Along this sub-sequence that we re-denote as $(n)$, we have 
\begin{align*} 
\lim_n \bs{\tilde\chi}_{m,n}^{\text{pert}}
 &= \lim_n  
 \frac 1m \EE \log \EGibbs{\int_{\Delta^\varepsilon_m(D) \cap [0,a]^m}  
 \exp\left( 
 \ps{u}{Q(\tx)} + \beta\alpha h^m \ps{u}{1_m} 
 \right) \mu_\beta^{\otimes m}(du)}' \nonumber  \\
 &\phantom{=} 
  - \lim_n \frac 1m \EE \log \EGibbs{ \exp\left( \sqrt{m} Y(\tx) \right) }' 
   - (h^m)^2 \frac{\beta\alpha}{2} . 
\end{align*}
Furthermore, it is known that we can approximate the limit distribution of 
the overlap $\widetilde R_{12}$ with a measure $\zeta^m \in \fop_D$. By
absorbing the approximation error into, \emph{e.g.}, the small number 
$\varepsilon$, we obtain that 
\[
\liminf_n \bs{\tilde\chi}_{m,n}^{\text{pert}} \geq 
f^1_m(\Delta^\varepsilon_m(D),a,\zeta^m,h^m) - f^2(\zeta^m) 
- \frac{\beta\alpha}{2} (h^m)^2 - \varepsilon .
\]
Now, consider a sub-sequence of $(m)$ converging to infinity along which 
$h^m$ converges to a real number $h^\infty$, $\zeta^m$ converges 
narrowly, and $f^1_m(\Delta^\varepsilon_m(D),a,\zeta^m,h^m) - f^2(\zeta^m) 
- \frac{\beta\alpha}{2} (h^m)^2$ converges to its limit inferior. 
By Lemma~\ref{f12lip}, there exists $\zeta^\infty \in \fop_D$ such that 
\begin{align*}
\liminf_m f^1_m(\Delta^\varepsilon_m(D),a,\zeta^m,h^m) - f^2(\zeta^m) 
- \frac{\beta\alpha}{2} (h^m)^2 &\geq 
\liminf_m f^1_m(\Delta^\varepsilon_m(D),a,\zeta^\infty,h^\infty) 
 - f^2(\zeta^\infty) - \frac{\beta\alpha}{2} (h^\infty)^2 - \varepsilon
  \\
&\geq 
 \inf_{\gamma} \left( \Phi_a(\zeta^\infty,h^\infty,\gamma) - \gamma D \right) 
 - f^2(\zeta^\infty) - \frac{\beta\alpha}{2} (h^\infty)^2 - \varepsilon \\
&= \inf_{\gamma\in\RR} \left( P_a(\zeta^\infty, h^\infty,\gamma) 
    - \gamma D \right) - \frac{\beta\alpha}{2} (h^\infty)^2 - \varepsilon, 
\end{align*} 
where the second inequality is due to Lemma~\ref{pan-lm6}. This leads to the 
bound
\[
\liminf_n F_{a,n}^{\Delta^\varepsilon_n(D)} \geq 
  \inf_{\zeta\in\fop_D, h\geq 0,\gamma\in\RR} 
 \left( P_a(\zeta, h,\gamma) - \gamma D - \frac{\beta\alpha}{2} h^2 \right) 
   - C \varepsilon^{1/4} .
\]
We therefore have 
\[
\liminf F_{a,n} \geq 
\sup_{D\in[0,a^2]} \liminf F_{a,n}^{\Delta^\varepsilon_n(D)} \geq 
 \sup_{D\in[0,a^2]} \inf_{\zeta\in\fop_D, h\geq 0,\gamma} 
\left( P_a(\zeta, h,\gamma) - \gamma D - \beta\alpha h^2/2 \right) 
  - C \varepsilon^{1/4}. 
\]
Since $\varepsilon$ is arbitrary, we obtain that 
$\liminf F_{a,n} \geq 
  \sup_{D\in[0,a^2]} \inf_{\zeta\in\fop_D, h\geq 0,\gamma} 
 \left( P_a(\zeta, h,\gamma) - \gamma D - \beta\alpha h^2/2 \right)$, which
shows by Lemma~\ref{F>Fa} that 
$\liminf_n \widetilde F_n \geq \sup_{a \geq 0} 
  \sup_{D\in[0,a^2]} g(a,D)$ with 
\[
g(a,D) = \inf_{\zeta\in\fop_D, h\geq 0,\gamma} 
 \left( P_a(\zeta, h,\gamma) - \gamma D - \beta\alpha h^2/2 \right).
\]
It remains to show that 
$\liminf_n \widetilde F_n \geq \sup_{a \geq 0, D\geq 0} g(a,D)$ to obtain
Proposition~\ref{low-bnd}. If $\sup_{a \geq 0} \sup_{D\in[0,a^2]} g(a,D)$
is attained for $a=\infty$, the result is obvious. Otherwise, assume that
this supremum is attained for some $\bar a < \infty$, and assume, furthermore,
that there exists $\bar D > \bar a^2$ such that 
$g(\bar a, \bar D) > \sup_{D\in[0,\bar a^2]} g(\bar a,D)$. Since 
$g(\cdot, \bar D)$ is increasing, we obtain that 
$g(\sqrt{\bar D}, \bar D) > \sup_{D\in[0,\bar a^2]} g(\bar a,D) = 
\sup_{a \geq 0} \sup_{D\in[0,a^2]} g(a,D)$, a contradiction. 
Proposition~\ref{low-bnd} is proven. \\

Theorem~\ref{F->P} for $\alpha\leq 0$ results from Propositions~\ref{up-bnd} 
and~\ref{low-bnd}.

\section{Proof of Theorem~\ref{F->P} for $\alpha > 0$.} 
\label{alpha>0} 
We still use the notation $\eta_n(x) = \ps{x}{1} / n$. 
The term $\beta\alpha n \eta_n(x)^2 / 2$ in the Hamiltonian $H_n(x)$ is 
reminiscent of the ferro-magnetic interaction that appears in, \emph{e.g.}, 
the Curie-Weiss model. The proof principle for dealing with this term when
$\alpha > 0$ is well known, and can be found in \cite{chen-14}  
in the SK case with ferro-magnetic interaction. 
\begin{lemma}
Given a real number $h\in\RR$, consider the Hamiltonian with external field 
$\HEF_n(\cdot;h)$ defined on $\RR_+^n$ as 
\[
\HEF_n(x;h) = 
 \frac{\beta\kappa}{2\sqrt{n}} x^\T W x + \beta\alpha h n \, \eta_n(x) . 
\]
Given $a, A > 0$, define the free energies $F^{\text{EF},A}_n(h)$ and  
$\FEF_{a,n}(h)$ as

\begin{align*}
F^{\text{EF},A}_n(h) &= 
 \frac 1n \EE\log \int_{B_+^n(\sqrt{n} A)} \exp\left(\HEF_n(x;h)\right) 
  \mu_\beta^{\otimes n}(dx), \quad \text{and} \\
\FEF_{a,n}(h) &= \frac 1n \EE\log \int_{[0,a]^n} \exp\left(\HEF_n(x;h)\right) 
  \mu_\beta^{\otimes n}(dx). 
\end{align*} 
Then, 
\begin{align} 
\limsup_n F^{\text{EF},A}_n(h) &\leq 
  \sup_{a,D \geq 0} \inf_{\zeta\in\fop_D,\gamma\in\RR} 
\left(   P_a(\zeta,h,\gamma) - \gamma D \right), 
  \quad \text{and} \label{FEFA} \\
\lim_n \FEF_{a,n}(h) &=   
  \sup_{D \in[0, a^2]} \inf_{\zeta\in\fop_D,\gamma\in\RR} 
\left(   P_a(\zeta,h,\gamma) - \gamma D \right) . 
 \label{FEFa} 
\end{align} 
\end{lemma} 
\begin{proof}
We can write 
\[
F^{\text{EF},A}_n(h) = \frac 1n \EE\log \int_{B_+^n(\sqrt{n} A)} 
\exp\left( \frac{\beta\kappa}{2\sqrt{n}} x^\T W x \right) 
  \nu_{\beta,h}^{\otimes n}(dx) , 
\]
with $\nu_{\beta,h}$ being the positive measure on $\RR_+$ defined as 
\[
\nu_{\beta,h}(dx_1) = 
 x_1^{\phi\beta - 1} 
  \exp\left( -\beta x_1^2 / 2 + \beta (1+\alpha h) x_1 \right) \, d x_1.
\]
Therefore, we can apply to this free energy the development of
Section~\ref{sec:ubnd} after replacing the measure $\mu_\beta$ with
$\nu_{\beta,h}$, and considering that $\alpha = 0$ in the Hamiltonian.  This
leads to the bound~\eqref{FEFA} which is the analogue of~\eqref{lisupFA}.  We
stress that in~\eqref{FEFA}, we  take the infimum over $\zeta\in\fop_D$ and
$\gamma\in\RR$ only because there is no term in $\eta_n(x)^2$ in the
Hamiltonian $\HEF(\cdot;h)$. 
 
To establish the convergence~\eqref{FEFa}, write 
\[
\FEF_{a,n}(h) = \frac 1n \EE\log \int_{[0,a]^n} 
\exp\left( \frac{\beta\kappa}{2\sqrt{n}} x^\T W x \right) 
  \nu_{\beta,h}^{\otimes n}(dx) , 
\]
and use \cite[Theorem 1]{pan-18}. 
\end{proof} 

We shall work with $\liminf \widetilde F_n$ then with $\limsup \widetilde F_n$.

\subsection{Lower bound} 
\label{low-alpha>0} 
Considering the free energy $F_{a,n}$ used in the statement of 
Lemma~\ref{F>Fa}, and using that $\eta(x)^2 \geq 2h \eta(x) - h^2$ for an 
arbitrary $h\in\RR$, we have 
\[
F_{a,n} = \frac 1n \EE\log \int_{[0,a]^n} \exp\left( 
 \frac{\beta\kappa}{2\sqrt{n}} x^\T W x + 
         \frac{\beta\alpha n}{2} \eta(x)^2 \right) \mu_\beta^{\otimes n}(dx) 
 \geq \FEF_{a,n}(h) - \frac{\beta\alpha}{2} h^2 . 
\]
By the previous lemma, 
\[
\liminf F_{a,n} \geq \sup_{h\geq 0, D\in[0,a^2]} 
  \inf_{\zeta\in\fop_D, \gamma\in\RR} 
  \left( P_a(\zeta,h,\gamma) - \gamma D - \frac{\beta\alpha}{2} h^2\right) , 
\]
and we obtain by Lemma~\ref{F>Fa} that 
\begin{align*} 
\liminf \widetilde F_n &\geq 
\sup_{a\geq 0, D\in[0,a^2], h\geq 0} \inf_{\zeta\in\fop_D, \gamma\in\RR} 
  \left( P_a(\zeta,h,\gamma) - \gamma D - \frac{\beta\alpha}{2} h^2 \right) 
  \\
 &=  \sup_{a, D, h\geq 0} \inf_{\zeta\in\fop_D, \gamma\in\RR} 
  \left( P_a(\zeta,h,\gamma) - \gamma D - \frac{\beta\alpha}{2} h^2 \right) 
\end{align*} 
by the argument developed at the end of the proof of 
Proposition~\ref{low-bnd} above.

\subsection{Upper bound} 
Write the free energy $F^A_n$ in the statement of Lemma~\ref{F<F^A} 
as $F_n^A = \EE \cX^A_n$ where $\cX^A_n$ is the so-called free energy density. 
Given a measurable bounded set $\cS \subset\RR_+^n$, write 
\[
F^\cS_n = \frac 1n \EE\log \int_\cS e^{H_n(x)} \mu_\beta^{\otimes n}(dx) 
 \quad\text{and}\quad 
\cX^\cS_n = \frac 1n \log \int_\cS e^{H_n(x)} \mu_\beta^{\otimes n}(dx) 
\]
(in such a way that $F^A_n = F^{B_+^n(\sqrt{n} A)}_n$ and 
$\cX^A_n = \cX^{B_+^n(\sqrt{n} A)}_n$). Fix a large number $N > 0$ 
independently of $n$. For an integer $i > 0$, consider the 
$\ell^1$ ring $\cR_i \subset\RR_+^n$ defined as 
\[
\cR_i = \{ x \in \RR_+^n \ : \ (i-1)/N < \eta_n(x) \leq i/N \}. 
\]
By Cauchy-Schwarz, each $x\in B_+^n(\sqrt{n} A)$ satisfies 
$\eta_n(x) \leq A$. Therefore, writing $M = \lceil NA \rceil$, we have 
that $\cR_1 \cap B_+^n(\sqrt{n} A), \ldots, \cR_M \cap B_+^n(\sqrt{n} A)$ is 
a partition of $B_+^n(\sqrt{n} A)$. Thus, 
\[
\cX^A_n \leq \frac{\log M}{n} + 
  \max_{i\in[M]} \cX_n^{\cR_i \cap B_+^n(\sqrt{n} A)} . 
\]
By Gaussian concentration, see, \emph{e.g.}, \cite[Theorem~1.2]{pan-livre13},  
it holds that  
\[
\forall i \in [M], \forall t > 0, \quad 
\PP\left[ \left| \cX^{\cR_i \cap B_+^n(\sqrt{n} A)}_n - 
   F^{\cR_i \cap B_+^n(\sqrt{n} A)}_n \right| \geq t \right] 
 \leq 2 \exp( - n t^2 / (2\beta^2\kappa^2 A^2) ),  
\]
which implies that 
$\EE \left(\cX^{\cR_i \cap B_+^n(\sqrt{n} A)}_n - 
F_n^{\cR_i \cap B_+^n(\sqrt{n} A)}\right)^2 \leq 4 \beta^2\kappa^2 A^2  / n$. 
Consequently, 
\[
F^A_n = \EE \cX^A_n \leq 
 \frac{\log M}{n} + \EE \max_{i\in[M]} \cX^{\cR_i \cap B_+^n(\sqrt{n} A)} 
\leq \frac{\log M}{n} + \frac{CM}{\sqrt{n}} 
 + \max_{i\in[M]} F_n^{\cR_i \cap B_+^n(\sqrt{n} A)}. 
\]
When $\eta \in ( (i-1)/N, i/N]$, it holds that 
$\eta^2 \leq 2\eta i/N - (i/N)^2 + 1/N^2$. We can thus write 
\begin{align*} 
F_n^{\cR_i \cap B_+^n(\sqrt{n} A)} &= \frac 1n \EE\log 
\int_{\cR_i \cap B_+^n(\sqrt{n} A)} \exp\left( 
 \frac{\beta\kappa}{2\sqrt{n}} x^\T W x + 
         \frac{\beta\alpha n}{2} \eta(x)^2 \right) \mu_\beta^{\otimes n}(dx) \\
 &\leq F^{\text{EF},A}_n(i/N) 
  - \frac{\beta\alpha}{2} \left(\frac iN \right)^2 
  + \frac{\beta\alpha}{2N^2} , 
\end{align*} 
and we obtain by the previous lemma that 
\begin{align*} 
\limsup_n F^A_n &\leq \max_{i\in[M]} \left(\limsup_n \FEF_n(i/N)
  - \frac{\beta\alpha}{2} \left(\frac iN \right)^2  \right) 
  + \frac{\beta\alpha}{2N^2}  \\ 
 &\leq \sup_{a,h,D \geq 0} \inf_{\zeta\in\fop_D,\gamma\in\RR} 
\left(   P_a(\zeta,h,\gamma) - \gamma D - \frac{\beta\alpha}{2} h^2 \right) 
  + \frac{\beta\alpha}{2N^2} . 
\end{align*} 
Recalling that $N$ is arbitrarily large and using Lemma~\ref{F<F^A}, we obtain
that 
\[
\limsup_n \widetilde F_n \leq 
 \sup_{a,h,D \geq 0} \inf_{\zeta\in\fop_D,\gamma\in\RR} 
\left(   P_a(\zeta,h,\gamma) - \gamma D - \frac{\beta\alpha}{2} h^2 \right) , 
\]
and Theorem~\ref{F->P} is proven for $\alpha > 0$. 

\section{Proof of Proposition~\ref{P_infty}} 

We first deal with the continuity issue of the Parisi functional with respect
to the measure $\zeta$.  The proof of the following lemma is close to the proof
of Lemma~\ref{f12lip} and is omitted. 
\begin{lemma}
\label{Xcont} 
Given $\zeta, \tilde\zeta \in \fop_D$, it holds that 
\[
 \left| X_{0,a}(\zeta,h,\gamma) 
  - X_{0,a}(\tilde\zeta,h,\gamma) \right| 
  \leq C \bd(\zeta,\tilde\zeta) , 
\]
where the constant $C$ depends on $a$ only and where the distance $\bd$ is 
defined in~\eqref{defdist}. 
\end{lemma}

Since $P_a(\zeta,h,\gamma) = X_{0,a}(\zeta,h,\gamma) +0.5 \int \theta(b)
\zeta(db) - 0.5 \theta(D)$, we obtain from this lemma and the Dominated
Convergence Theorem that $P_a(\zeta,h,\gamma)$ can be continuously extended
from $\fop_D\times\RR_+\times\RR$ to $\cP([0,D])\times\RR_+\times\RR$, and the
first part of the proposition is proven. To establish the second part, we
distinguish between the cases $\alpha\leq 0$ and $\alpha > 0$.

\subsection{Proposition~\ref{P_infty}: convergence for $\alpha\leq 0$}

Our first step is to show that $X_{0,a}(\zeta,h,\gamma)$ is a convex function.
To this end, we start by characterizing $X_{0,a}$ through a Parisi PDE (Partial
Differential Equation). Let $\zeta\in\fop_D$. Defining the function 
$g_a : \RR\times\RR_+\times\RR \to \RR$ as 
\[
  g_a(v,h,\gamma) = \log \int_0^a  \exp \left( xv + \beta\alpha h x 
  + \gamma x^2 \right) \mu_\beta(dx) , 
\]
our PDE reads 
\begin{align*} 
&\partial_t \Phi(t,v) + \frac{\xi''(t)}{2} \left( \partial_{vv}^2 \Phi(t,v) + 
 \zeta([0,t]) (\partial_v \Phi'(t,v))^2 \right) = 0, 
   \quad (t,v) \in (0,D) \times \RR,  \\
&\Phi(D,v) = g_a(v,h,\gamma) 
\end{align*} 
Getting back to the recursive construction starting from $X_{K,a}$ and 
leading to $X_{0,a}$, and using the so-called Cole-Hopf transformation,
we know that 
\[
X_{0,a} = \Phi(0,0),  
\]
see, \emph{e.g.}, \cite[Chapter 6]{dom-mou-livre24} for a more detailed 
treatment of this PDE. 

In~\cite{auf-chen-15} and~\cite{jag-tob-16}, the PDE solution is given a
variational form, which can be verified without difficulty in our case.  Define
as $\cA$ the class of real-valued bounded random processes which are
progressively measurable on $[0,D]$ with respect to a Brownian Motion $B_t$ on
$[0,D]$.  Then, it holds that 
\begin{equation}
\label{prog} 
 X_{0,a}(\zeta,h,\gamma) = \sup_{f\in\cA} \EE\left[ 
   -\frac 12 \int_0^1 \xi''(t) \zeta([0,t]) f(t)^2 dt 
    + g_a(Z_D^f,h,\gamma) \right], 
\end{equation} 
where the random process $(Z_t^f)_{t\in[0,D]}$ solves the SDE 
\[
dZ_t^f = \xi''(t) \zeta([0,t]) f(t) dt + \sqrt{\xi''(t)} dB_t. 
\]
Let us quickly check that $g_a(v,h,\gamma)$ is convex.  Write $u = (v,h,\gamma)
\in \RR\times\RR_+\times\RR$.  For a given $u$, let $\EGibbs{\cdot}$ be the
expectation operator for the probability measure $G$ supported by $[0,a]$, and
that satisfies $G(dx) \sim \exp \left( xv + \beta\alpha h x + \gamma x^2
\right) \mu_\beta(dx)$, and let $r(x) = [x , \beta\alpha x, x^2]^\T \in \RR^3$.
Then, we see that $\nabla g_a(u) = \EGibbs{r(x)}$ and $\nabla^2 g_a(u) =
\EGibbs{r(x) r(x)^\T} - \EGibbs{r(x)} \EGibbs{r(x)}^\T \geq 0$, which shows
that $g_a$ is convex.  With this at hand, the convexity of
$X_{0,a}(\zeta,h,\gamma)$ is obtained through a straightforward adaptation of
the beginning of the proof of \cite[Theorem~20]{jag-tob-16}, relying on the
variational representation~\eqref{prog} of $X_{0,a}(\zeta,h,\gamma)$. 

We know that the function $X_{0,a}(\zeta,h,\gamma)$ can be continuously
extended from $\fop_D \times \RR_+\times\RR$ to $\cP([0,D]) \times
\RR_+\times\RR$.  Furthermore, by the previous result, this extension that we
still denote as $X_{0,a}$ is convex on $\cP([0,D]) \times \RR_+\times\RR$.
Recall that we denoted as $P_a$ the extension of $P_a$ to $\cP([0,D]) \times
\RR_+\times\RR$. Write $Q_a(\zeta,h,\gamma) = P_a(\zeta,h,\gamma) - \beta\alpha
h^2 / 2$.  The function $Q_a$ is convex on $\cP([0,D]) \times \RR_+\times\RR$
since $X_{0,a}$ is convex, the term $\int \theta d\zeta$ is linear and the term
$- \beta\alpha h^2 / 2$ is convex for $\alpha\leq 0$. The function $Q_a$ is
moreover continuous. 

Given $\zeta\in\cP([0,D])$, define the sequence 
$\bm^\zeta = \left( \int_0^D x^k \zeta(dx) / k! \right)_{k\geq 1}$ which 
belongs to $\ell^2(\NN)$. Define the function 
$S : \cP([0,D])\times\RR_+\times\RR \to \ell^2(\NN)$ as 
\[
S((\zeta,h,\gamma)) = \left( h, \gamma, \bm^\zeta\right) 
\]
where the right hand side is meant to be the $\ell^2(\NN)$ sequence obtained 
by preceding the sequence $\bm^\zeta$ with $(h,\gamma)$. The function $S$ 
is an injection from $\cP([0,D])\times\RR_+\times\RR$ to the set 
$\cD = S(\cP([0,D])\times\RR_+\times\RR)$ since each element of $\cP([0,D])$ is
determined by its moments. Define the function $\mathfrak Q_a$ as  
\[
\begin{array}{lcccl} 
\mathfrak Q_a &:& \ell^2(\NN) & \to & \RR\cup\{\infty\} \\
            & &  f & \mapsto & 
  \left\{\begin{array}{cl}  Q_a(S^{-1}(f)) &\text{if} \ f\in \cD \\ 
                        \infty &\text{otherwise.} 
         \end{array}\right.  
\end{array} 
\]
This function is proper. Moreover, it is convex since its domain $\cD$ is 
convex, $S^{-1}(uf_1 + (1-u)f_2) = u S^{-1}(f_1) + (1-u) S^{-1}(f_2)$ for 
each $u\in[0,1]$, and $Q_a$ is convex. 

The convergence in $\ell^2(\NN)$ of the elements of the set $\cD_\cP = \{
\bm^\zeta \ : \ \zeta\in \cP([0,D]) \}$ is equivalent to the finite dimensional
convergence. Furthermore, since $\cP([0,D])$ is a compact space and since the
narrow convergence in this space is equivalent to the convergence of the
moments, the set $\cD_\cP$ is a compact. This implies that $\cD$ is a closed
subset of $\ell^2(\NN)$. Given $c\in\RR$, let $\lev_{\leq c} \mathfrak
Q_a$ be the $c$--level set of $\mathfrak Q_a$, assumed non-empty, and
let $(f_k)$ be a sequence in $\lev_{\leq c} \mathfrak Q_a$ that converges
to $f\in\ell^2(\NN)$. Since $\cD$ is closed, $f\in\cD$. By the continuity of
$\bQ_a$, $f\in \lev_{\leq c} \mathfrak Q_a$. Thus, $\mathfrak Q_a$ is a
lower semicontinuous function on $\ell^2(\NN)$ for each $a\in(0,\infty]$.

Write an element $f\in\ell^2(\NN)$ as $f = (h,\gamma,(m_k)_{k\geq 1})$. 
By the continuity of $\bQ_{a}$, we have 
\begin{align*} 
- \inf_{\zeta\in\fop_D, h\in\RR_+,\gamma\in\RR} 
 Q_a(\zeta, h,\gamma) - \gamma D &= 
 \sup_{\zeta\in\cP([0,D]), h\in\RR_+,\gamma\in\RR} 
 \gamma D - Q_a(\zeta, h,\gamma) \\ 
&= \sup_{(h,\gamma,(m_k)) \in \ell^2(\NN)} \gamma D 
   - \mathfrak Q_a((h,\gamma,(m_k))) \\ 
 &= (\mathfrak Q_a)^* ((0,D, (0,0,\ldots))) , 
\end{align*} 
where $(\mathfrak Q_a)^*$ is the Fenchel-Legendre transform of $\mathfrak
Q_a$.  The family of proper, convex, and lower semicontinuous functions
$(\mathfrak Q_a)_{a\in(0,\infty]}$ is increasing with $a$ and converges in
the pointwise sense to $\mathfrak Q_\infty$. Therefore, this convergence
takes place in the so-called Mosco (or epigraphic) sense, see
\cite{att-livre84} for a detailed account on this convergence.  In this
situation, it is well-known that $(\mathfrak Q_a)^* \to_a (\mathfrak
Q_\infty)^*$ in the Mosco sense \cite{mos-71}, or equivalently, since
$(\mathfrak Q_a)^*$ is decreasing in $a$, in a pointwise sense. Consequently, 
\begin{align*} 
\sup_{a > 0} \inf_{\zeta\in\fop_D, h\in\RR_+,\gamma\in\RR}  
 P_a(\zeta, h,\gamma) - \gamma D - \beta\alpha h^2 / 2 &= 
\lim_{a\to\infty} \inf_{\zeta\in\fop_D, h\in\RR_+,\gamma\in\RR} 
 Q_a(\zeta, h,\gamma) - \gamma D \\ 
 &= - (\mathfrak Q_\infty)^* ((0,D, (0,0,\ldots))) \\ 
 &= \inf_{\zeta\in\cP([0,D]), h\in\RR_+,\gamma\in\RR} 
 Q_\infty(\zeta, h,\gamma) - \gamma D  \\ 
 &= \inf_{\zeta\in\cP([0,D]), h\in\RR_+,\gamma\in\RR}  
 P_\infty(\zeta, h,\gamma) - \gamma D - \beta\alpha h^2 / 2 , 
\end{align*} 
and the convergence of $\widetilde F_n$ in the statement of 
Proposition~\ref{P_infty} for $\alpha\leq 0$ is deduced from the first 
convergence in the statement of Theorem~\ref{F->P}.

\subsection{Proposition~\ref{P_infty}: convergence for $\alpha > 0$}
Here, we just notice that the function 
$(\zeta,\gamma) \mapsto P_a(\zeta,h,\gamma)$ is convex. By adapting the 
previous argument, we obtain that 
\[
\sup_{a>0} \inf_{\zeta\in\fop_D,\gamma\in\RR} 
  P_a(\zeta,h,\gamma) - \gamma D) = 
 \inf_{\zeta\in\cP([0,D]),\gamma\in\RR} 
  P_\infty(\zeta,h,\gamma) - \gamma D). 
\]
Combined with the second convergence in the statement of Theorem~\ref{F->P}, 
this result leads to the second convergence stated by Proposition~\ref{P_infty}.


\appendix
\section{Proof of Proposition~\ref{prop-eds}} 
\label{anx-eds}  

Using the martingale representation theorem, we know from, \emph{e.g.},
\cite[Chapter~4, Theorem~2.2 and 2.3 and Remark.~2.1]{ike-wat-livre89} that for each
probability measure $\mu \in \cP(\RR_+^n)$, there exists a weak solution to the
SDE~\eqref{eds} such that $x_0 \sim \mu$ and $x_0 \indep B$. Defining the
explosion time of the process $(x_t)$ as 
\[
\tau_\infty = \inf \{ t \geq 0 \ : \ \| x_t \| = \infty \}, 
\]
we now show that \eqref{l+} (or, equivalently, $\lambda_+^{\min} > 0$) is
satisfied, then $\tau_\infty = \infty$ with probability one, which means
that the solutions of~\eqref{eds} never explode. 

To this end, it is enough to assume that $x_0$ is an arbitrary deterministic
vector in $\RR_+^n$.  Writing $V(x) = \ps{1}{x}$, any solution of the
SDE~\eqref{eds} starting with $x_0$ satisfies 
\begin{align*}
V(x_t) &= V(x_0) + \int_0^t \left( \ps{x_u}{1} + \phi n 
  + \ps{x_u}{\left( \Sigma - I \right) x_u} \right) du 
   + \sqrt{2 T} \int_0^t \ps{\sqrt{x_u}}{dB_u} .
\end{align*}
Given $a > 0$, define the stopping time 
\[
\tau_a = \inf \left\{ t \geq 0 \ : \ V(x_t) \geq a \right\} . 
\]
Observing that $\tau_\infty$ is a $\bar\RR$--valued random variable given as
$\tau_\infty = \lim_{a\to\infty} \tau_a$, our purpose is to show that
$\PP[\tau_\infty = \infty] = 1$.  The techniques for establishing this
convergence are well-known \cite{mao-mar-ren-02}.  In our case, we write 
\begin{align*} 
V(x_{t\wedge\tau_a}) &= V(x_0) + \int_0^{t\wedge\tau_a}
   \left( \ps{x_u}{1} + \phi n 
  + \ps{x_u}{\left( \Sigma - I \right) x_u} \right) du 
   + \sqrt{2 T} \int_0^{t\wedge\tau_a} \ps{\sqrt{x_u}}{dB_u},
\end{align*}
thus, 
\begin{align*} 
\EE V(x_{t\wedge\tau_a}) &= V(x_0) + \EE \int_0^{t\wedge\tau_a}
   \left( \ps{x_u}{1} + \phi n 
  - \ps{x_u}{\left( I - \Sigma \right) x_u} \right) du \\ 
 &\leq V(x_0) + \EE \int_0^{t\wedge\tau_a} 
  \left( C - \lambda_+^{\min} \| x_u \|^2 / 2 \right) du \\
&\leq V(x_0) + C \EE[t\wedge\tau_a].  
\end{align*}
Assume there exists $T, \varepsilon > 0$ such that $\PP[\tau_\infty
\leq T] \geq \varepsilon$, which implies that 
$\PP[\tau_a \leq T] \geq \varepsilon$ for all $a$. Setting $t = T$, we get that 
\[
\EE V(x_{T\wedge\tau_a}) \leq V(x_0) + C \EE (T\wedge\tau_a) \leq 
 V(x_0)  + CT .
\]
On the event $\cE_a = [\tau_a \leq T]$, we have $V(x_{\tau_a}) = a$. Therefore,
\[
V(x_0)  + CT \geq \EE V(x_{T\wedge\tau_a}) \geq 
 \EE \1_{\cE_a} V(x_{\tau_a}) \geq \varepsilon a .
\]
Making $a\to\infty$, we obtain the contradiction 
\[
V(x_0)  + CT \geq \infty .
\]
The last step of the proof is to establish the pathwise uniqueness of the
solution of~\eqref{eds}. Indeed, pathwise uniqueness implies the existence of a
unique strong solution for~\eqref{eds} \cite[Chapter~4, Theorem~1.1]{ike-wat-livre89}. 

For notational simplicity, we rewrite Eq.~\eqref{eds} as $dx_t = f(x_t) dt +
\sigma(x_t) dB_t$ with $f$ being the vector function $f(x) = [ f_i(x)
]_{i\in[n]} = x(1 + (\Sigma-I)x) + \phi$ and $\sigma(x) = \sqrt{2T x}$.  Let
$x^1_t = [ x_{i,t}^1 ]_{i\in[n]}$ and $x^2_t = [ x_{i,t}^2 ]_{i\in[n]}$ be two
solutions starting at the same point $x_0$ and defined with the same BM $B_t = [
B_{i,t} ]_{i\in[n]}$. Observing that the function $\sigma$ satisfies the
inequality $|\sigma(x) - \sigma(y)| \leq \rho(|x-y|)$ with $\rho(x) = \sqrt{2T}
\sqrt{x}$ satisfying $\int_0^1 \rho^{-2}(x) dx = \infty$, we can use the
technique of the proof of \cite[Chapter~4, Theorem~3.2]{ike-wat-livre89} to construct
a sequence of $\RR\to\RR_+$ functions $(\varphi_k)_{k\geq 1}$ such that
$\varphi_k \in \cC^2(\RR;\RR)$, $\varphi_k(x) \uparrow |x|$ as $k\to\infty$,
$|\varphi_k'(x) | \leq 1$, and $0\leq \varphi_k''(x) \leq 2 \rho^{-2}(x) / k$. 

Let $\Delta_{i,t} = x^1_{i,t} - x^2_{i,t}$, consider the SDE 
\[
\begin{bmatrix} d x^1_t \\ d x^2_t \end{bmatrix} = 
\begin{bmatrix} f(x^1_t) \\ f(x^2_t) \end{bmatrix} dt +
\begin{bmatrix} \diag\sigma(x^1_t) \\ \diag\sigma(x^2_t) \end{bmatrix} 
 dB_t, 
\]
and, denoting as $\|\cdot\|_1$ the $\ell_1$ norm in $\RR^n$, define the 
stopping time 
\[
\eta_a = \inf \left\{ t \geq 0 \ : \ 
 \| x^1_t \|_1 \vee \| x^2_t \|_1 \geq a \right\} 
\]
for $a > 0$. Applying Itô's formula to the SDE above, we obtain 
\begin{align*}
\sum_i \varphi_k(\Delta_{i,t\wedge\eta_a}) &= \int_0^{t\wedge\eta_a} \sum_i 
  \varphi_k'(\Delta_{i,u}) \left( f_i(x^1_u) - f_i(x^2_u) \right) du 
  + \frac 12 \int_0^{t\wedge\eta_a} \sum_i 
 \varphi_k''(\Delta_{i,u}) 
\left( \sigma(x^1_{i,u}) - \sigma_i(x^2_u) \right)^2 du  \\
&\phantom{=} 
  + \int_0^{t\wedge\eta_a} \sum_i 
 \varphi_k'(\Delta_{i,u}) 
\left( \sigma(x^1_{i,u}) - \sigma_i(x^2_u) \right) dB_{i,u}  . 
\end{align*} 
Observing that the function 
$b : (\RR_+^n, \|\cdot\|_1) \to (\RR^n,\|\cdot\|_1)$ is 
Lipschitz on the ball $\{ x \in \RR_+^n \ : \ \| x \|_1 \leq a \}$ with 
the Lipschitz constant $C_a$, we obtain 
\begin{align*}
 \sum_i \EE \varphi_k(\Delta_{i,t\wedge\eta_a}) &= 
 \EE \int_0^{t\wedge\eta_a} \sum_i 
  \varphi_k'(\Delta_{i,u}) \left( f_i(x^1_u) - f_i(x^2_u) \right) du 
  + \frac 12 \EE \int_0^{t\wedge\eta_a} \sum_i 
 \varphi_k''(\Delta_{i,u}) 
\left( \sigma(x^1_{i,u}) - \sigma_i(x^2_{i,u}) \right)^2 du  \\
 &\leq C_a \EE \int_0^{t\wedge\eta_a} \sum_i \left| \Delta_{i,u} \right| du 
  + \frac nk \EE \int_0^{t\wedge\eta_a} du  \\
 &\leq C_a \EE \int_0^{t\wedge\eta_a} \sum_i \left| \Delta_{i,u} \right| du 
  + \frac{nt}{k}. 
\end{align*} 
By taking $k\to\infty$, we then obtain by monotone convergence that 
\[
\sum_i \EE \left| \Delta_{i,t\wedge\eta_a} \right| \leq 
  C_a \EE \int_0^{t\wedge\eta_a} \sum_i \left| \Delta_{i,u} \right| du 
 \leq C_a \int_0^{t} \sum_i \EE \left| \Delta_{i,u\wedge\eta_a} \right| du . 
\]
Using Grönwall's lemma, we obtain that 
$\sum_i \EE | \Delta_{i,t\wedge\eta_a} | = 0$, thus, 
$\sum_i \EE | \Delta_{i,t} | = 0$, which shows that $x^1$ and $x^2$ are
indistinguishable by continuity. 

\section{Proof of Lemma~\ref{l-wd}}
\label{prf-l-wd} 
We first need some properties of the functions $f$ and $g$, which can be 
obtained by a small adaptation of \cite[Lemma 18]{mon-ric-16}. These read: 
The function $f$ is strictly positive and differentiable.  It satisfies
$f'(x)>0$ for all $x\in \RR$, $\lim_{x\to-\infty} f(x) = 0$,  $f(0) =
1/\sqrt{\pi}$, and $\lim_{x\to +\infty} f(x) = 1$.  The function $g$ is a
strictly positive differentiable function that satisfies $g'(x) > 0$ for all $x
< 0$, $g'(0) = 0$, and $g'(x)<0$ for all $x > 0$. Also, $\lim_{|x| \to \infty}
g(x) = 0$.

We first observe that if $\alpha = 0$, then trivially, $c = 0$ and it is the
unique maximizer of $r_0$ by the aforementioned properties of $g$. Assume
$\alpha \neq 0$. 

Let us define the real function function $q$ on $\RR\setminus\{0\}$ as
$q(x) = f(x) / x$. On this domain of definition, it holds that 
\[
q'(x) = (xf'(x) - f(x))/x^2 . 
\]
Therefore, it is clear that $q$ is strictly decreasing on $(-\infty, 0)$. 
We also know from the proof of \cite[Lemma 20]{mon-ric-16} that $q$ is also
strictly decreasing on $(0,\infty)$. Finally, it is obvious from what 
precedes that $\lim_{|x|\to \infty} q(x) = 0$, 
$\lim_{x\to 0^-} q(x) = -\infty$, and $\lim_{x\to 0^+} q(x) = +\infty$. 
It results that the equation $x = \alpha f(x)$ has an unique solution 
$c \in\RR$ for each $\alpha\in\RR\setminus\{0\}$. 

It remains to show that $c$ is the unique maximum of $r_\alpha$.
A small calculation shows that $(\sqrt{d(x)})' = f(x)$ and that 
$\sqrt{d(x)} = x f(x) + g(x)$. Equating the derivatives at both sides of this
last equation shows that $g'(x) = - x f'(x)$. We can therefore write 
\[
r_\alpha'(x) = 2\alpha f(x)f'(x) + 2 g'(x)
 = 2x f'(x) \left(\alpha \frac{f(x)}{x} - 1\right) = 
 2x f'(x) \left(\alpha q(x) - 1\right)  . 
\]
Let us focus here on the case where $\alpha < 0$. The case $\alpha > 0$ is
similar. From the previous results on the graph of $q$, we obtain that $c < 0$ 
and that
\[  
 \left\{\begin{array}{ll} 
x(\alpha q(x) - 1) > 0  &\text{if } x < c \\
x(\alpha q(x) - 1) = 0  &\text{if } x = c \\
x(\alpha q(x) - 1) < 0  &\text{if } x > c. \end{array}\right.
\]
Recalling that $f' > 0$, we obtain that $c$ is the unique maximizer of 
$r_\alpha$. 

\section{Proof of Lemma~\ref{F<F^A}} 
\label{prf-F<F^A} 
Define the event 
\[
\cE = \left\{ \lambda_+^{\max}(\Sigma) < 1 - \varepsilon_\Sigma  \right\} , 
\]
and recall that $\1_{\cE^{\text{c}}} \to_n 0$ almost surely by 
Proposition~\ref{bunin-bd}. Given a number $a > 0$, define the sets $\mathfrak C_+(a) \subset \RR_+^n$ and 
$\mathfrak C(a) \subset \RR^n$ as 
\[
\mathfrak C_+(a) = \left\{ x = (x_i) \in \RR_+^n, \ \forall i \in[n], 
 x_i \geq a \right\} \quad \text{and} \quad 
\mathfrak C(a) = \left\{ x = (x_i) \in \RR^n, \ \forall i \in[n], 
 |x_i| \geq a \right\} .
\]
Let $B^n(a)$ be the closed ball of $\RR^n$ with radius $a$, and let 
$Z \sim \cN(0, I_n)$. 

Let $A > 0$. 
Whether we set $\bH(x) = \widetilde \mcH(x)$ or $\bH(x) = \mcH(x)$, obtain by
inspecting the expressions of the Hamiltonians $\widetilde\mcH$ and $\mcH$ that 
\begin{align*}
\int_{B_+(\sqrt{n} A)} e^{\beta \bH(x)} dx &\geq 
\int_{B_+(\sqrt{n} A) \cap \mathfrak C_+(1)} 
     e^{-\beta(\|\Sigma\| +1) \| x\|^2/2} dx 
= 2^{-n} \int_{B(\sqrt{n} A) \cap \mathfrak C(1)} 
    e^{-\beta(\|\Sigma\| +1) \| x\|^2/2} dx  \\
&= \left(\frac{\pi}{2\beta(\|\Sigma\| +1)}\right)^{n/2} 
 \PP_Z\left[ Z \in B\left(\sqrt{n \beta(\|\Sigma\|+1)} A\right) \cap 
  \mathfrak C\left(\sqrt{\beta(\|\Sigma\|+1)}\right) \right]   \\
&\geq \left(\frac{\pi}{2\beta(\|\Sigma\| +1)}\right)^{n/2} 
 \left( \PP_Z\left[ Z \in \mathfrak C\left(\sqrt{\beta(\|\Sigma\|+1)}
     \right)\right]  
 - \PP_Z\left[ Z \in B\left(\sqrt{n \beta(\|\Sigma\|+1)} A\right)^{\text{c}} 
  \right] \right) . 
\end{align*} 
Using, \emph{e.g.}, \cite[7.1.13]{abr-ste-64} to lower bound the Gaussian
tail function, we obtain that there exists a constant $C_\beta$ depending on
$\beta$ only such that
$\PP_Z\left[ Z \in \mathfrak C\left(\sqrt{\beta(\|\Sigma\|+1)}
     \right)\right] \geq C_\beta^n \exp(-n\beta(\|\Sigma\| + 1))$. By Gaussian
concentration, we also have that 
$\PP_Z\left[ Z \in B\left(\sqrt{n \beta(\|\Sigma\|+1)} A\right)^{\text{c}} 
  \right] \leq \exp(-n\beta(\|\Sigma\| +1)A^2/2)$ for $A$ large enough, which
is negligible with respect to $C_\beta^n \exp(-n\beta(\|\Sigma\| + 1))$ for
large $A$. Putting things together, we obtain that 
\begin{equation}
\label{eH>} 
\int_{B_+(\sqrt{n} A)} e^{\beta \bH(x)} dx \geq 
C^n (\|\Sigma\| + 1)^{-n/2} \exp(-n\beta(\|\Sigma\| + 1)) 
\end{equation} 
for $A$ large enough. 

Now, setting $B_+(\sqrt{n} A)^{\text{c}} = \RR_+^n \setminus B_+(\sqrt{n} A)$, 
we write 
\begin{align*} 
\widetilde F &= 
\frac 1n \EE \1_\cE \log \left( \int_{B_+(\sqrt{n} A)} e^{\beta \widetilde\mcH} 
  + \int_{B_+(\sqrt{n} A)^{\text{c}}} e^{\beta \widetilde\mcH} \right)  
 + \frac 1n \EE \1_{\cE^{\text{c}}} 
  \log \int_{\RR_+^n} e^{\beta \widetilde\mcH} \\
&\eqdef  
\frac 1n \EE \1_\cE \log \left( \cI_{B_+(\sqrt{n} A)} + 
 \cI_{B_+(\sqrt{n} A)^{\text{c}}} \right)  
 + \frac 1n \EE \1_{\cE^{\text{c}}} 
  \log \int_{\RR_+^n} e^{\beta \widetilde\mcH} . 
\end{align*} 
To manage the term with the indicator $\1_{\cE^{\text{c}}}$, we observe that 
$\widetilde\Sigma = 0$ on the event $\cE^{\text{c}}$.  With this, the integral
in this third term is deterministic and can be easily shown to satisfy
$n^{-1}\log\int_{\RR_+^n} \exp(\beta\widetilde\mcH) \leq C$.  Therefore, this
term is negligible because $\PP[\cE^{\text{c}}] \to 0$. 
 
We now manage the terms $\cI_{B_+(\sqrt{n} A)}$ and $\cI_{B_+(\sqrt{n}
A)^{\text{c}}}$, essentially showing that the latter is negligible with respect
to the former.  On the event $\cE$, it holds that $x^\T (\Sigma - I )x \leq -
\varepsilon_\Sigma \| x \|^2$ on $\RR_+^n$, thus, there exists a constant $c =
c(\phi,\beta)$ such that on this event, 
\begin{align*} 
\cI_{B_+(\sqrt{n} A)^{\text{c}}} 
 &\leq 
\int_{B_+(\sqrt{n} A)^{\text{c}}} \, 
   e^{-\frac\beta 2 (\varepsilon_\Sigma \| x \|^2 - 2 c\ps{1}{x})} dx 
= e^{\frac{\beta c^2}{2\varepsilon_\Sigma} n }  
\int_{B_+(\sqrt{n} A)^{\text{c}}} \, 
   e^{-\frac\beta 2 \left\| \sqrt{\varepsilon_\Sigma} x 
   - \frac{c}{\sqrt{\varepsilon_\Sigma}} 1 \right\|^2} dx . 
\end{align*} 
By making the variable change 
$u = \sqrt{\varepsilon_\Sigma} x- \frac{c}{\sqrt{\varepsilon_\Sigma}} 1$ and
noticing that $\| x \| \geq \sqrt{n} A \Rightarrow 
 \| u \| \geq \sqrt{n \varepsilon_\Sigma} A - c\sqrt{n / \varepsilon_\Sigma}$,
we obtain that 
\[
\cI_{B_+(\sqrt{n} A)^{\text{c}}} \leq 
 \varepsilon_\Sigma^{-n/2} e^{\frac{\beta c^2}{2\varepsilon_\Sigma} n }  
\int_{B(\sqrt{n} 
  (\sqrt{\varepsilon_\Sigma} A - c / \sqrt{\varepsilon_\Sigma}))^{\text{c}}} \, 
   e^{-\beta \| u \|^2 / 2} du 
\]
on $\cE$ for $A$ large enough, where $B(a)^{\text{c}} = \RR^n \setminus B(a)$. 
By Gaussian concentration, we finally get that 
\[
\cI_{B_+(\sqrt{n} A)^{\text{c}}} \leq \exp(- nC (A^2 - 1)) 
\]
on $\cE$ for $A$ large enough. 

We now write 
\[
\frac 1n \EE \1_\cE \log \left( 
\cI_{B_+(\sqrt{n} A)} + \cI_{B_+(\sqrt{n} A)^{\text{c}}} \right)  
= 
\frac 1n \EE \1_\cE \log \cI_{B_+(\sqrt{n} A)} + 
\frac 1n \EE \1_\cE \log \left( 1 + \cI_{B_+(\sqrt{n} A)}^{-1}  
 \cI_{B_+(\sqrt{n} A)^{\text{c}}} \right) . 
\]
Using Inequality~\eqref{eH>}, we have 
\[
\frac 1n \EE \1_\cE \log \left( 1 + \cI_{B_+(\sqrt{n} A)}^{-1}  
 \cI_{B_+(\sqrt{n} A)^{\text{c}}} \right) 
\leq \frac 1n \EE \log \left( 1 + e^{nC(\| \Sigma \| + 1 - A^2)} \right), 
\]
and since $\log(1+e^x) \leq e^x \1_{x< 0} + (x+1) \1_{x\geq 0}$, we can write
\[
\frac 1n \EE \log \left( 1 + e^{nC(\| \Sigma \| + 1 - A^2)} \right)
\leq \frac 1n + C\EE (\| \Sigma \| + 1 - A^2) \1_{\| \Sigma \| \geq A^2 - 1} . 
\]
Using, \emph{e.g.}, \cite[Corollary 4.4.8]{ver-livre18}, we know that 
$\EE \| \Sigma \|^2 \leq 2 \kappa^2 \EE \| W \|^2 + 2 \alpha^2 < C$. Thus, for
$A$ large enough, the right hand side converges to zero by the 
Cauchy-Schwarz inequality and the standard results on the behavior of 
$\| \Sigma \| = \|\kappa W + \alpha 1 1^\T / n \|$ for large $n$. 

Getting back to the expression of $\widetilde F_n$ provided above, we then
obtain that 
\[
\widetilde F_n = \frac 1n \EE \1_\cE \log \cI_{B_+(\sqrt{n} A)} + o_n(1).
\]

Turning to $F_n^A$, we write 
\[
F_n^A = \frac 1n \EE \1_{\cE} \log \int_{B_+(\sqrt{n} A)} e^{\beta \mcH} 
 + \frac 1n \EE \1_{\cE^{\text{c}}} \log \int_{B_+(\sqrt{n} A)} e^{\beta \mcH} 
= \frac 1n \EE \1_\cE \log \cI_{B_+(\sqrt{n} A)} 
 + \frac 1n \EE \1_{\cE^{\text{c}}} \log \int_{B_+(\sqrt{n} A)} e^{\beta \mcH} 
\]
since $\Sigma = \widetilde\Sigma$ on the event $\cE$. 
By \eqref{eH>}, using that $\log(1+|a|) \leq |a|$, we obtain that 
the second term satisfies 
\[
\frac 1n \EE \1_{\cE^{\text{c}}} \log \int_{B_+(\sqrt{n} A)} e^{\beta \mcH} 
\geq - C \EE \1_{\cE^{\text{c}}} (1 + \| \Sigma \|). 
\]
By the Cauchy-Schwarz inequality, the right hand side is negligible. 

We therefore have that $F_n^A - \widetilde F_n \geq o_n(1)$ for $A$ large
enough, and Lemma~\ref{F<F^A} is proven. 

\section{Proof of Lemma~\ref{F>Fa}}
\label{prf-F>Fa}  

Still writing $\cE = \left\{ \lambda_+^{\max}(\Sigma) < 1 - 
  \varepsilon_\Sigma \right\}$, we have 
\[
F_{a,n} = \frac 1n \EE \1_{\cE} \log \int_{[0,a]^n} e^{\beta \mcH(x)} dx 
 + \frac 1n \EE \1_{\cE^{\text{c}}} \log \int_{[0,a]^n} e^{\beta\mcH(x)} dx  
 \eqdef \chi_{1,n} + \chi_{2,n}.  
\] 
We bound the term $\chi_{2,n}$ by writing 
\begin{align*}
\chi_{2,n} &\leq 
 \frac 1n \EE \1_{\cE^{\text{c}}} \log 
  \int_{[0,a]^n} \exp\left(\beta \| \Sigma - I \| \| x \|^2 /2 
  + \beta\ps{1}{x} + (\beta\phi - 1) \ps{1}{\log x} \right) \ dx \\
 &= \EE \1_{\cE^{\text{c}}} \log \int_0^a 
 \exp \left( \beta \| \Sigma - I \| x^2 /2 + \beta x + (\beta\phi - 1) \log x 
  \right) \ dx 
 \\
 &\leq \EE  \1_{\cE^{\text{c}}} \log \left( a 
  \exp \left( \beta \| \Sigma - I \| a^2/2 + \beta a + (\beta\phi - 1) \log a 
  \right) \right) \\
&\leq \EE \1_{\cE^{\text{c}}} \left( \beta\phi\log a 
  + \beta a^2 \| \Sigma - I \| / 2 + \beta a \right) 
\end{align*} 
which goes to zero as $n\to\infty$ by using 
\cite[Corollary 4.4.8]{ver-livre18} and Cauchy-Schwarz. We also have 
\[
\widetilde F_n \geq  
\frac 1n \EE \1_\cE \log \int_{[0,a]^n} e^{\beta \mcH} 
+ \frac 1n 
  \EE \1_{\cE^{\text{c}}} \log \int_{\RR_+^n} e^{\beta \widetilde\mcH}. 
\] 
As in Appendix~\ref{prf-F<F^A}, the second term at the right hand side is 
negligible. This proves Lemma~\ref{F>Fa}. 

\section*{Acknowledgement}
\noindent We would like to thank one of the referees for her/his remarks that 
substantially improved our paper. 

\section*{Data availability and funding} 
\noindent The authors did not receive support from any organization for the submitted work, except for the PhD funding of the first author by Université Gustave Eiffel. There is no data associated with this manuscript.
\section*{Conflict of interest}
\noindent The authors declare that they have no conflicts of interest.

\bibliographystyle{plain}
\bibliography{math}

\def\cdprime{$''$} \def\cprime{$'$} \def\cprime{$'$} \def\cprime{$'$} \def\cprime{$'$}
\begin{thebibliography}{10}

\bibitem{abr-ste-64}
M.~Abramowitz and I.~A. Stegun.
\newblock {\em Handbook of Mathematical Functions with Formulas, Graphs, and Mathematical Tables}, volume~55 of {\em National Bureau of Standards Applied Mathematics Series}.
\newblock United States Department of Commerce, 1964.

\bibitem{akj-etal-24}
I.~Akjouj, M.~Barbier, M.~Clenet, W.~Hachem, M.~Ma{\"\i}da, F.~Massol, J.~Najim, and V.-C. Tran.
\newblock Complex systems in ecology: a guided tour with large {L}otka--{V}olterra models and random matrices.
\newblock {\em Proceedings of the Royal Society A}, 480(2285):20230284, 2024.

\bibitem{all-tan-12}
S.~Allesina and S.~Tang.
\newblock Stability criteria for complex ecosystems.
\newblock {\em Nature}, 483(7388):205--208, 2012.

\bibitem{alt-22}
A.~Altieri.
\newblock Glassy features and complex dynamics in ecological systems.
\newblock {\em arXiv preprint arXiv:2208.14956}, 2022.

\bibitem{alt-roy-cam-bir-21}
A.~Altieri, F.~Roy, C.~Cammarota, and G.~Biroli.
\newblock Properties of equilibria and glassy phases of the random {L}otka-{V}olterra model with demographic noise.
\newblock {\em Phys. Rev. Lett.}, 126:258301, Jun 2021.

\bibitem{att-livre84}
H.~Attouch.
\newblock {\em Variational convergence for functions and operators}.
\newblock Applicable Mathematics Series. Pitman (Advanced Publishing Program), Boston, MA, 1984.

\bibitem{auf-chen-15}
A.~Auffinger and W.-K. Chen.
\newblock The {P}arisi formula has a unique minimizer.
\newblock {\em Comm. Math. Phys.}, 335(3):1429--1444, 2015.

\bibitem{bir-bun-cam-18}
G.~Biroli, G.~Bunin, and C.~Cammarota.
\newblock Marginally stable equilibria in critical ecosystems.
\newblock {\em New Journal of Physics}, 20(8):083051, aug 2018.

\bibitem{bun-17}
G.~Bunin.
\newblock Ecological communities with {L}otka-{V}olterra dynamics.
\newblock {\em Phys. Rev. E}, 95(4):042414, 8, 2017.

\bibitem{chen-14}
W.-K. Chen.
\newblock On the mixed even-spin {Sherrington-Kirkpatrick} model with ferromagnetic interaction.
\newblock {\em Annales de l'I.H.P. Probabilit\'es et statistiques}, 50(1):63--83, 2014.

\bibitem{dom-mou-livre24}
T.~Dominguez and J.-C. Mourrat.
\newblock {\em Statistical Mechanics of Mean-Field Disordered Systems}.
\newblock Zurich Lectures in Advanced Mathematics. EMS Press, 2024.

\bibitem{gue-01}
F.~Guerra.
\newblock Sum rules for the free energy in the mean field spin glass model.
\newblock In {\em Mathematical physics in mathematics and physics ({S}iena, 2000)}, volume~30 of {\em Fields Inst. Commun.}, pages 161--170. Amer. Math. Soc., Providence, RI, 2001.

\bibitem{gue-03}
F.~Guerra.
\newblock Broken replica symmetry bounds in the mean field spin glass model.
\newblock {\em Comm. Math. Phys.}, 233(1):1--12, 2003.

\bibitem{ike-wat-livre89}
N.~Ikeda and Sh. Watanabe.
\newblock {\em Stochastic differential equations and diffusion processes}, volume~24 of {\em North-Holland Mathematical Library}.
\newblock North-Holland Publishing Co., Amsterdam; Kodansha, Ltd., Tokyo, second edition, 1989.

\bibitem{jag-tob-16}
A.~Jagannath and I.~Tobasco.
\newblock A dynamic programming approach to the {P}arisi functional.
\newblock {\em Proc. Amer. Math. Soc.}, 144(7):3135--3150, 2016.

\bibitem{kha-livre12}
R.~Khasminskii.
\newblock {\em Stochastic stability of differential equations}, volume~66 of {\em Stochastic Modelling and Applied Probability}.
\newblock Springer, Heidelberg, second edition, 2012.
\newblock With contributions by G. N. Milstein and M. B. Nevelson.

\bibitem{mao-11}
X.~Mao.
\newblock Stationary distribution of stochastic population systems.
\newblock {\em Systems Control Lett.}, 60(6):398--405, 2011.

\bibitem{mao-mar-ren-02}
X.~Mao, G.~Marion, and E.~Renshaw.
\newblock Environmental {B}rownian noise suppresses explosions in population dynamics.
\newblock {\em Stochastic Process. Appl.}, 97(1):95--110, 2002.

\bibitem{mon-ric-16}
A.~Montanari and E.~Richard.
\newblock Non-negative principal component analysis: message passing algorithms and sharp asymptotics.
\newblock {\em IEEE Trans. Inform. Theory}, 62(3):1458--1484, 2016.

\bibitem{mos-71}
U.~Mosco.
\newblock On the continuity of the {Y}oung-{F}enchel transform.
\newblock {\em J. Math. Anal. Appl.}, 35:518--535, 1971.

\bibitem{pan-05}
D.~Panchenko.
\newblock Free energy in the generalized {S}herrington-{K}irkpatrick mean field model.
\newblock {\em Rev. Math. Phys.}, 17(7):793--857, 2005.

\bibitem{pan-livre13}
D.~Panchenko.
\newblock {\em The {S}herrington-{K}irkpatrick model}.
\newblock Springer Monographs in Mathematics. Springer, New York, 2013.

\bibitem{pan-18}
D.~Panchenko.
\newblock Free energy in the mixed {$p$}-spin models with vector spins.
\newblock {\em Ann. Probab.}, 46(2):865--896, 2018.

\bibitem{pan-potts-18}
D.~Panchenko.
\newblock Free energy in the {P}otts spin glass.
\newblock {\em Ann. Probab.}, 46(2):829--864, 2018.

\bibitem{par-79}
G.~Parisi.
\newblock Infinite number of order parameters for spin-glasses.
\newblock {\em Phys. Rev. Lett.}, 43:1754--1756, Dec 1979.

\bibitem{par-80}
G.~Parisi.
\newblock A sequence of approximated solutions to the {S-K} model for spin glasses.
\newblock {\em Journal of Physics A: Mathematical and General}, 13(4):L115, apr 1980.

\bibitem{roy-etal-19}
F.~Roy, G.~Biroli, G.~Bunin, and C.~Cammarota.
\newblock Numerical implementation of dynamical mean field theory for disordered systems: Application to the {L}otka--{V}olterra model of ecosystems.
\newblock {\em Journal of Physics A: Mathematical and Theoretical}, 52(48):484001, 2019.

\bibitem{tal-(sph)-06}
M.~Talagrand.
\newblock Free energy of the spherical mean field model.
\newblock {\em Probab. Theory Related Fields}, 134(3):339--382, 2006.

\bibitem{tal-06}
M.~Talagrand.
\newblock The {P}arisi formula.
\newblock {\em Ann. of Math. (2)}, 163(1):221--263, 2006.

\bibitem{tal-livre11-t1}
M.~Talagrand.
\newblock {\em Mean field models for spin glasses. {V}olume {I}}, volume~54 of {\em Ergebnisse der Mathematik und ihrer Grenzgebiete. 3. Folge. A Series of Modern Surveys in Mathematics [Results in Mathematics and Related Areas. 3rd Series. A Series of Modern Surveys in Mathematics]}.
\newblock Springer-Verlag, Berlin, 2011.
\newblock Basic examples.

\bibitem{tal-livre11-t2}
M.~Talagrand.
\newblock {\em Mean field models for spin glasses. {V}olume {II}}, volume~55 of {\em Ergebnisse der Mathematik und ihrer Grenzgebiete. 3. Folge. A Series of Modern Surveys in Mathematics [Results in Mathematics and Related Areas. 3rd Series. A Series of Modern Surveys in Mathematics]}.
\newblock Springer, Heidelberg, 2011.
\newblock Advanced replica-symmetry and low temperature.

\bibitem{ver-livre18}
R.~Vershynin.
\newblock {\em High-Dimensional Probability: An Introduction with Applications in Data Science}.
\newblock Cambridge Series in Statistical and Probabilistic Mathematics. Cambridge University Press, 2018.

\end{thebibliography}

\end{document}